\documentclass[reqno,twoside,11pt]{amsart}
\usepackage{amsmath}
\usepackage{amsfonts}
\usepackage{amssymb}
\usepackage{bbm}
\usepackage{graphicx}
\usepackage{verbatim}

\IfFileExists{mathptmx.sty}{\usepackage{mathptmx}}{\usepackage{mathpazo}}
\usepackage{mathrsfs}

\DeclareFontFamily{OT1}{eusb}{} \DeclareFontShape{OT1}{eusb}{m}{n} {<5> <6> <7> <8> <9> <10> <11> <12> <14.4> eusb10}{}
\DeclareMathAlphabet{\eusb}{OT1}{eusb}{m}{n}

\DeclareFontFamily{OT1}{eusm}{} \DeclareFontShape{OT1}{eusm}{m}{n} {<5> <6> <7> <8> <9> <10> <11> <12> <14.4> eusm10}{}
\DeclareMathAlphabet{\eusm}{OT1}{eusm}{m}{n}

\DeclareFontFamily{OT1}{eufm}{} \DeclareFontShape{OT1}{eufm}{m}{n} {<5> <6> <7> <8> <9> <10> <11> <12> <14.4> eufm10}{}
\DeclareMathAlphabet{\mathfrak}{OT1}{eufm}{m}{n}

\DeclareFontFamily{OT1}{fraktura}{}
\DeclareFontShape{OT1}{fraktura}{m}{n} {<5> <6> <7> <8> <9> <10> <11> <12> <13> <14.4> [1.1] eufm10}{}
\DeclareMathAlphabet{\fraktura}{OT1}{fraktura}{m}{n}

\DeclareFontFamily{OT1}{cmfi}{} \DeclareFontShape{OT1}{cmfi}{m}{n} {<5> <6> <7> <8> <9> <10> <11> <12> <13> <14.4> [0.9] cmfi10}{}
\DeclareMathAlphabet{\cmfi}{OT1}{cmfi}{b}{n}

\DeclareFontFamily{OT1}{cmss}{} \DeclareFontShape{OT1}{cmss}{m}{n} {<5> <6> <7> <8> <9> <10> <11> <12> <13> <14.4> cmss10}{}
\DeclareMathAlphabet{\cmss}{OT1}{cmss}{m}{n}

\newtheoremstyle{thm}{1.5ex}{1.5ex}{\itshape\rmfamily}{} {\bfseries\rmfamily}{}{2ex}{}

\newtheoremstyle{def}{1.5ex}{1.5ex}{\rmfamily\sl}{} {\bfseries\rmfamily}{}{2ex}{}

\newtheoremstyle{rem}{1.3ex}{1.3ex}{\rmfamily}{} {\bfseries\rmfamily}{}{2ex}{}

\newtheoremstyle{ass}{1.5ex}{1.5ex}{\rmfamily\sl}{} {\bfseries\rmfamily}{}{2ex}{}

\newenvironment{proofsect}[1] {\vskip0.1cm\noindent{\rmfamily\itshape#1.}}{\qed\vspace{0.15cm}}

\theoremstyle{thm}
\newtheorem{theorem}{Theorem}[section]
\newtheorem{lemma}[theorem]{Lemma}
\newtheorem{proposition}[theorem]{Proposition}

\newtheorem*{Main Theorem}{Main Theorem.}
\newtheorem{corollary}[theorem]{Corollary}
\newtheorem{conjecture}[theorem]{Conjecture}

\theoremstyle{def}

\theoremstyle{rem}
\newtheorem{remark}[theorem]{{Remark}}

\numberwithin{equation}{section}

\renewcommand{\section}{\secdef\sct\sect}
\newcommand{\sct}[2][default]{\refstepcounter{section}
\addcontentsline{toc}{section}
{{\tocsection {}{\thesection}{\!\!\!\!#1\dotfill}}{}}
\vspace{0.7cm}
\centerline{ 
\scshape\arabic{section}.\ #1} \nopagebreak \vspace{0.2cm}}
\newcommand{\sect}[1]{
\vspace{0.4cm} \centerline{\large\scshape\rmfamily #1}
\vspace{0.2cm}}

\renewcommand{\subsection}{\secdef\subsct\sbsect}
\newcommand{\subsct}[2][default]{\refstepcounter{subsection}
\addcontentsline{toc}{subsection}
{{\tocsection{\!\!}{\hspace{1.2em}\thesubsection}{\!\!\!\!#1\dotfill}}{}}
\nopagebreak\vspace{0.45\baselineskip} {\flushleft\bf
\arabic{section}.\arabic{subsection}~\bf #1.~}
\\*[3mm]\noindent
\nopagebreak}
\newcommand{\sbsect}[1]{\vspace{0.1cm}\noindent
\textfb{#1.~}\vspace{0.1cm}}

\renewcommand{\subsubsection}{%
\secdef \subsubsect\sbsbsect}
\newcommand{\subsubsect}[2][default]{%
\refstepcounter{subsubsection}
\addcontentsline{toc}{subsubsection}{{\tocsection{\!\!}
{\hspace{3.05em}\thesubsubsection}{\!\!\!\!#1\dotfill}}{}}
\nopagebreak
\vspace{0.15\baselineskip} \nopagebreak {\flushleft\rmfamily
\itshape\arabic{section}.\arabic{subsection}.\arabic{subsubsection}
\ \rmfamily #1\/.}\ }
\newcommand{\sbsbsect}[1]{\vspace{0.1cm}\noindent
\rmfamily \itshape
\arabic{section}.\arabic{subsection}.\arabic{subsubsection} \
\sffamily #1\/.\ }

\renewcommand{\caption}[1]{%
\vglue0.5cm
\refstepcounter{figure}
\begin{minipage}{0.9\textwidth}\small {\sc Figure~\thefigure. }#1\end{minipage}}

\newcommand{\diam}{\operatorname{diam}}

\newcommand{\textd}{\text{\rm d}\mkern0.5mu}
\newcommand{\texti}{\text{\rm  i}\mkern0.7mu}
\newcommand{\texte}{\text{\rm  e}\mkern0.7mu}

\newcommand{\Var}{\text{\rm Var}}
\newcommand{\Cov}{\text{\rm \Cov}}
\newcommand{\1}{{1\mkern-4.5mu\textrm{l}}}
\renewcommand{\1}{\text{\sf 1}}

\renewcommand{\AA}{\mathcal A}

\newcommand{\EE}{\mathcal E}
\newcommand{\FF}{\mathcal F}

\newcommand{\PP}{\mathcal P}

\newcommand{\RR}{\mathcal R}

\newcommand{\TT}{\mathcal T}
\newcommand{\UU}{\mathcal U}

\newcommand{\E}{\mathbb E}

\newcommand{\N}{\mathbb N}

\newcommand{\BbbP}{\mathbb P}
\newcommand{\Q}{\mathbb Q}
\newcommand{\R}{\mathbb R}

\newcommand{\Z}{\mathbb Z}

\newcommand{\scrF}{\mathscr{F}}

\newcommand{\frb}{\mathfrak b}
\newcommand{\frp}{\mathfrak p}

\newcommand{\twoeqref}[2]{(\ref{#1}--\ref{#2})}

\newcommand{\cc}{{\text{\rm c}}}

\def\myffrac#1#2 in #3{\raise 2.6pt\hbox{$#3 #1$}\mkern-1.5mu\raise 0.8pt\hbox{$#3/$}\mkern-1.1mu\lower 1.5pt\hbox{$#3 #2$}}
\newcommand{\ffrac}[2]{\mathchoice%
        {\myffrac{#1}{#2} in \scriptstyle}
        {\myffrac{#1}{#2} in \scriptstyle}
        {\myffrac{#1}{#2} in \scriptscriptstyle}
        {\myffrac{#1}{#2} in \scriptscriptstyle}
}

\newcommand{\mbC}{\mathbf{C}}
\newcommand{\intr}{\text{\rm int}}
\newcommand{\hull}{\text{\rm hull}}
\newcommand{\vol}{\text{\rm fill}}
\newcommand{\Leb}{\text{\rm Leb}}

\newcommand{\len}{\text{\rm len}}
\newcommand{\poly}{\text{\rm poly}}

\newcommand{\rmB}{{\rm B}}
\newcommand{\wt}{\widetilde}
\newcommand{\wh}{\widehat}
\newcommand{\bae}{\begin{equation}\begin{aligned}}
\newcommand{\eae}{\end{aligned}\end{equation}}

\newcommand{\ind}{{\mathbbm{1}}}

\newcommand{\eqd}{\overset{\textd}=}

\newcounter{podsekce}
\newcommand\podsekce[1]{\refstepcounter{podsekce}\vskip0.05cm\noindent{(\thepodsekce)\itshape\ #1:}\ }

\begin{document}


\title[Isoperimetry in Supercritical Percolation]
{\large 
Isoperimetry in two-dimensional percolation}
\author[M.~Biskup, O.~Louidor, E.B.~Procaccia and R.~Rosenthal]
{M.~Biskup$^{1,2}$, O.~Louidor$^{1}$, E.B.~Procaccia$^{3}$ and R.~Rosenthal$^{4}$}
\thanks{\hglue-4.5mm\fontsize{9.6}{9.6}\selectfont\copyright\,\textrm{2012} \textrm{M.~Biskup, O.~Louidor, E.B.~Procaccia and R.~Rosenthal.
Reproduction, by any means, of the entire
article for non-commercial purposes is permitted without charge.\vspace{2mm}}}
\maketitle

\vspace{-5mm}
\centerline{\textit{$^1$Department of Mathematics, UCLA, Los Angeles, California, USA}}
\centerline{\textit{$^2$School of Economics, University of South Bohemia, \v Cesk\'e Bud\v ejovice, Czech Republic}}
\centerline{\textit{$^3$Faculty of Mathematics and Computer Science, Weizmann Institute of Science, Rehovot, Israel}}
\centerline{\textit{$^4$Einstein Institute of Mathematics, Hebrew University, Jerusalem, Israel}}

\vspace{-0mm}
\begin{abstract}
We study isoperimetric sets, i.e.,  sets with minimal boundary for a prescribed volume, on the unique infinite connected component of supercritical bond percolation on the square lattice. In the limit of the volume tending to infinity, properly scaled isoperimetric sets are shown to converge (in the Hausdorff metric) to the solution of an isoperimetric problem in~$\R^2$  with respect to a particular norm. As part of the proof we also show that the anchored isoperimetric profile as well as the Cheeger constant of the giant component in finite boxes scale to deterministic quantities. This settles a conjecture of Itai Benjamini for the square lattice.
\end{abstract}


\section{Introduction and Results}
\vglue-0.2cm
\subsection{Motivation}\noindent
Isoperimetry is a subject that lies at the heart of geometric measure theory. It provides a fundamental link between metric structures and measures on the underlying space. Isoperimetric inequalities have served as an essential tool for many analytical results. Indeed, they play a crucial role for subjects such as concentration of measure, Nash and Sobolev inequalities, 
spectra of Laplacians (Faber-Krahn and Poincar\'e inequalities), heat-kernel estimates, elliptic PDEs, mixing bounds for diffusions/random walks, etc. Isoperimetric problems, i.e., the characterization of sets of a 
prescribed volume with minimal boundary measure, have been around since the inception of modern science. Attempts for their solution lay the foundations for important methods in mathematics; e.g., the calculus of variations.

The classical isoperimetric problems were stated for the continuum but they have recently found their way into discrete mathematics as well (see, e.g., Chung~\cite[Chapter~2]{Chung-book}). For a finite graph $G=(V,E)$, isoperimetry is often characterized by the \emph{Cheeger constant}
\begin{equation}
\label{eqn:Cheeger}
\Phi_{G}:=
    \min\biggl\{\frac{|\partial_G U|}{|U|}\colon U\subset V,\,0<|U|\le\frac12|V|\biggr\},
\end{equation}
where $\partial_G U$ denotes the edge-boundary of $U$ in $G$, i.e., the set of edges in $E$ with exactly one endpoint in~$U$. The name owes its origin to the thesis of Cheeger~\cite{Cheeger-thesis}, where the bound $\lambda_1\ge c\Phi_G^2$ 
was derived for the first nonzero eigenvalue $\lambda_1$ of the negative Laplacian. (Cheeger's work deals with manifolds; for graph versions and connections to Markov chains see, e.g., Varopoulos~\cite{Varopoulos}, Lawler and Sokal~\cite{Lawler-Sokal}, Sinclair and Jerrum~\cite{Jerrum-Sinclair}.)
In computer science, the problem of finding the Cheeger constant of a graph is known as the {\em sparsest cut}.

When $G$ is infinite 
and amenable, then $\Phi_G=0$ by definition and so \eqref{eqn:Cheeger} is not very useful. A number of surrogates can be substituted instead; for our purposes the most interesting is the \emph{anchored isoperimetric profile}. Given a vertex $0 \in V$, to be called an \emph{anchor}, the isoperimetric profile of $G$ anchored at $0$ is the function $\Phi_{G, 0} \colon \R_+ \to \R_+$ given by
\begin{equation}
\label{eqn:Profile}
\Phi_{G, 0}(r) :=\inf\biggl\{\frac{|\partial_G U|}{|U|}
    \colon 0 \in U \subset V,\, G(U) \text{ \rm connected},\, 0<|U| \le r \biggr\},
\end{equation}
where $G(U)$ is the restriction of $G$ to vertices in~$U$. 
Isoperimetric profiles have proved to be instrumental for delicate mixing-time and heat-kernel estimates for Markov chains (see, e.g., Lov\'asz-Kannan~\cite{Lovasz-Kannan} or the books by Levin, Peres and Wilmer~\cite{Book-LPW} and Montenegro and Tetali~\cite{Montenegro-Tetali}). 

Solutions of isoperimetric problems, i.e., the minimizers in (\ref{eqn:Cheeger}--\ref{eqn:Profile}) are called {\em isoperimetric sets}. In graphs with an underlying geometrical structure, such as lattices, the isoperimetric sets can sometimes be characterized also geometrically. For instance, on~$\Z^d$, they correspond to balls in~$\ell^\infty$-metric (i.e., square boxes). However, aside from a few examples where the underlying geometry is simple and regular, describing isoperimetric sets is a difficult task.

In the present paper we study isoperimetric sets on graphs arising from bond percolation on~$\Z^d$; particularly, for~$d=2$. The subject of percolation is now considerably evolved, see Kesten~\cite{Kesten-percolation}, Grimmett~\cite{grimmett1999percolation}, Bollob\'as and Riordan~\cite{Bollobas-Riordan}, so we recount only the basic facts. Regard~$\Z^d$ as the graph with edge set $\EE^d$ given by all (unordered) nearest-neighbor pairs. Let~$\BbbP$ denote the product (Bernoulli) measure on $\{0,1\}^{\EE^d}$ with the density of~1's given by $p \in [0,1]$. An $\omega$ sampled from $\BbbP$ can be identified with a subgraph of~$\Z^d$, with edge set 
given by the edges~$e$ satisfying $\omega(e)=1$. These edges will be referred to as \emph{open}; 
edges $e$ with $\omega(e)=0$ will be called \emph{closed}. 

It is well known  that, for $d \geq 2$, there is a $p_\cc(d) \in (0,1)$ such that whenever $p > p_\cc(d)$,
the graph associated with $\omega$ contains a unique infinite connected component $\BbbP$-almost surely. This component, usually referred to as the \emph{infinite cluster}, will be denoted by  $\mbC^\infty$. Thanks to ergodicity of~$\BbbP$ with respect to the shifts, the asymptotic density of $\mbC^\infty$ in~$\Z^d$ is
\begin{equation}
\label{eqn:theta_p}
    \theta_p := \BbbP(0 \in \mbC^\infty),
\end{equation}
with $\theta_p>0$ for $p>p_\cc(d)$.
Similarly, for $p > p_\cc(d)$ with probability tending rapidly to $1$ as $n\to\infty$, 
the restriction of $\omega$ to the box $\rmB_\infty(n) := [-n, n]^d \cap \Z^d$, contains a unique connected component whose size is linear in $|\rmB_\infty(n)|$ and all other components are of size at most poly-logarithmic in~$n$. We shall denote this  \emph{giant component} by $\mbC^n$.

It is not hard to surmise the leading order of the anchored isoperimetric profile of~$\mbC^\infty$ and the Cheeger constant of $\mbC^n$ by invoking an analogy with the full lattice:
\begin{equation}
\label{E:1.4}
    \Phi_{\mbC^\infty, 0}(n) \asymp n^{-1/d} \quad\text{ and } \quad
    \Phi_{\mbC^n} \asymp n^{-1},\quad\qquad n\to\infty.
\end{equation}
Thanks to sophisticated facts from percolation theory, these bounds can be established with probability tending rapidly to one (Benjamini and Mossel \cite{benjamini2003mixing}, Mathieu and Remy \cite{mathieu2004isoperimetry}, Rau~\cite{Rau}, Berger, Biskup, Hoffman and Kozma \cite{berger2008anomalous} and Pete \cite{ECP1390}). This led Itai~Benjamini to formulate the following natural conjecture:

\begin{conjecture}
\label{con1}
For $p>p_\cc(\Z^d)$ (and $d\ge2$) the limit $\lim\limits_{n \to \infty} \,n\, \Phi_{\mbC^n}$ exists $\BbbP$\rm-a.s.
\end{conjecture}
\noindent

With this conjecture in sight, Procaccia and Rosenthal~\cite{procaccia2011concentration} have recently established the following bound: For $d\ge2$ and $p>p_\cc(d)$, there is $C=C(d,p)<\infty$ such that
\begin{equation}
\Var(n \Phi_{\mbC^n}) \leq C n^{2-d}.
\end{equation}
This implies concentration of $n \Phi_{\mbC^n}$ around its mean in all dimensions $\ge3$. Unfortunately, no information could be obtained about the limit of 
$\E(n \Phi_{\mbC^n}$). 

The principal aim of the present paper is to prove Benjamini's conjecture for the isoperimetric profile as well as a version of the Cheeger constant (see below) for supercritical bond percolation on~$\Z^2$. In addition, and perhaps more importantly, we characterize the asymptotic shape of the  minimizing sets via an isoperimetric problem in~$\R^2$
with respect to a particular norm.
The latter falls inside a class of variational problems whose solution is given by the so called 
``Wulff construction.'' This is a well-known term in statistical physics; indeed, it has surfaced before in the studies of droplet shapes in two-dimensional percolation (Alexander, Chayes and Chayes~\cite{ACC}) and the Ising model (Dobrushin, Koteck\'y, Shlosman~\cite{Book:DKS}, Ioffe and Schonmann~\cite{Ioffe-Schonmann}). See also Cerf~\cite{Cerf-book} and Bodineau, Ioffe and Velenik~\cite{BIV} for extensions to higher dimensions and general~overviews.  

In spite of certain similarities, there are substantial differences that make the above studies quite different from ours. First, the isoperimetric problem there arises from a large-deviation principle and thus shape-optimization occurs only at the continuum level. Second, the corresponding norm is defined there as the asymptotic decay rate of certain probabilities, while for us it comes from actual random variables. Third, in our context we thus have to control also the concentration of these random variables around their mean; a step that is not needed in~\cite{ACC,Book:DKS,Ioffe-Schonmann}. 

\subsection{Results}
We will now proceed to state the main conclusions of the present paper. All of our considerations will be limited to~$d=2$. Here we recall that, thanks to Kesten's Theorem~\cite{Kesten}, $p_\cc(\Z^2)=\ffrac12$. Notwithstanding, we will keep writing $p_\cc(\Z^2)$ as it is more illuminating. We begin with the isoperimetric profile where the result is easiest to formulate:

\begin{theorem}[Isoperimetric profile]
\label{thm:Profile}
Let $d = 2$ and $p > p_\cc(\Z^2)$.  Then there exists a constant $\varphi_p \in (0,\infty)$
(given by \eqref{eqn:IsoOfBeta} below)
such that $\BbbP(\cdot | 0 \in \mbC^\infty)$-almost surely,
\begin{equation}
\label{E:1.6}
    \lim_{n \to \infty} n^{1/2}\, \Phi_{\mbC^\infty,0}(n) = \theta_p^{-1/2} \varphi_p  ,
\end{equation}
where $\theta_p$ is defined in \eqref{eqn:theta_p}.
\end{theorem}

\begin{remark}
We write the limit value as a product of two (for now implicit) terms to emphasize their origin: $\theta_p$ arises from the volume restriction while $\varphi_p$ captures a boundary length.
\end{remark}

Next we will address the Cheeger constant for the giant component $\mbC^n$. As it turns out, it is more natural to look at the quantity $\widetilde\Phi_{\mbC^n, \mbC^\infty}$, where for a finite subgraph $G=(V,E)$ of a (possibly infinite) graph $H$, we define
\begin{equation}
\label{eqn:Cheeger'}
\widetilde\Phi_{G, H} :=
    \inf\biggl\{\frac{|\partial_{H} U|}{|U|}\colon U\subset V,\,0<|U|\le\frac12|V|\biggr\}.
\end{equation}
The rationale behind the use of $\widetilde\Phi_{\mbC^n, \mbC^\infty}$, or $\wt{\Phi}_{\mbC^n}$ for short, in place of $\Phi_{\mbC^n}$ is to avoid giving unfair advantage to sets which are ``attached'' to the boundary of $\rmB_\infty(n)$. We can now state a theorem which settles a version of Conjecture~\ref{con1}:

\begin{theorem}[Cheeger constant]
\label{thm:Cheeger}
Let $d = 2$ and $p > p_\cc(\Z^2)$ and let $\varphi_p$ be as in Theorem~\ref{thm:Profile} and~$\theta_p$ as in \eqref{eqn:theta_p}. Then,
$\BbbP$-almost surely,
\begin{equation}
\label{E:1.8}
    \lim_{n \to \infty} n \,\wt{\Phi}_{\mbC^n} = \frac1{\sqrt{2}} \theta_p^{-1} \varphi_p .
\end{equation}
\end{theorem}

\begin{remark}
The factor $1/\sqrt 2$ arises from the factor $\ffrac12$ in \eqref{eqn:Cheeger'} and the fact that $\rmB_\infty(n)$ has $\sim (2n)^2$ vertices. In particular, if $\ffrac12$ in \eqref{eqn:Cheeger'} is replaced by $\alpha\in(0,\ffrac12]$, then 
$1/\sqrt2$~is replaced by~$1/(2\alpha^{1/2})$~in \eqref{E:1.8}. 
\end{remark}

Having established the existence of a limit value, the next natural question is its characterization.  This requires an introduction of an isoperimetric problem in $\R^2$ for a specific norm. To set up the necessary notation, for a curve~$\lambda$, i.e., a continuous map $\lambda\colon[0,1]\to\R^2$, and a norm~$\rho$ on~$\R^2$, let the $\rho$-length of~$\lambda$ be defined as
\begin{equation}
\len_\rho(\lambda):=\sup_{N\ge1}\,\,\sup_{0\le t_0<\dots<t_N\le1}\,\,\sum_{i=1}^N\,\rho\bigl(\lambda(t_i)-\lambda(t_{i-1})\bigr).
\end{equation}
The curve~$\lambda$ is 
called rectifiable if $\len_\rho(\lambda)<\infty$ for any 
norm~$\rho$ on~$\R^2$. If~$\lambda$ is simple and closed (i.e., Jordan), its interior $\intr(\lambda)$ is the unique bounded component of~$\R^2\setminus\lambda$.

In Theorem~\ref{thm:BetaP} we will define a specific norm~$\beta_p$ associated with supercritical percolation with parameter~$p$. This norm  is invariant under reflections through the axes and the diagonals of~$\Z^2$. (The definition is rather involved and so we leave further specifics to Section~\ref{sec2}.) For this norm~$\varphi_p$ then solves a classical isoperimetric problem:

\begin{theorem}[Limit value]
\label{thm-limit-value}
Let $d = 2$ and $p > p_\cc(\Z^2)$ and let $\beta_p$ be the norm from Theorem~\ref{thm:BetaP}. Then $\varphi_p$ from Theorems~\ref{thm:Profile} and \ref{thm:Cheeger} satisfies
\begin{equation}
\label{eqn:IsoOfBeta}
\varphi_p = \inf \bigl\{\,\len_{\beta_p}(\lambda) \colon 
    \lambda \text{\, is a Jordan curve in~$\R^2$, }
    \Leb(\intr(\lambda)) = 1 \bigr\} \, .
\end{equation}
Here $\Leb$ stands for the Lebesgue measure on~$\R^2$.
\end{theorem}

Isoperimetric problems in $\R^d$ have a long history and much is known about them. In particular, thanks to observations made by Wulff~\cite{Wulff}, a minimizer of \eqref{eqn:IsoOfBeta} can be explicitly constructed. (This is what is referred to as the ``Wulff Construction.'') Define
\begin{equation}
\label{eqn:4.5}
    W_p := \bigcap_{\hat{n} \colon  \|\hat{n}\|_2 =1}
        \bigl\{x\in\R^2 \colon  \hat{n}\cdot x \leq \beta_p(\hat{n})\bigr\} ,
   \end{equation}
where $\hat n\cdot x$ is the Euclidean scalar product, and let
\begin{equation}
\label{E:1.12x}
\widehat{W}_p:=W_p / \sqrt{\Leb(W_p)} \,.
\end{equation}
Here and henceforth we adopt the notation (for $A\subset\R^2$, $\zeta\in\R$, $\xi\in\R^2$)
\begin{equation}
\zeta A:=\{\zeta x\colon x\in A\}\quad\text{and}\quad\xi+A:=\{\xi+x\colon x\in A\} .
\end{equation}
Hence, $\widehat{W}_p$ is $W_p$ normalized to have a unit area. Note that $W_p$ can be viewed as the unit ball in the dual norm $\beta_p'$, standardly defined for $y \in \R^2$ as
\begin{equation}
	\beta_p'(y) = \sup \{x \cdot y \colon  x \in \R^2,\, \beta_p(x) \leq 1\} .
\end{equation}
Since $\widehat W_p$ is a convex domain, its boundary is a simple curve, so we can set
\begin{equation}
\label{eqn:4.51}
    \hat{\gamma}_p := \partial \widehat{W}_p.
\end{equation}
Then $\hat{\gamma}_p$ is a minimizer of~\eqref{eqn:IsoOfBeta}. A proof using the Brunn-Minkowski inequality was given by Taylor in~\cite{Taylor-1}.

In Taylor~\cite{Taylor-2}, it is then shown that the minimizer is unique up to shifts. Let us write $\|x-y\|$ for the $\ell^\infty$-distance between~$x$ and~$y$ and~$d_{\rm H}$ for the $\ell^\infty$-Hausdorff metric on compact sets in~$\R^2$, 
\begin{equation}
\label{E:dH}
d_{\rm H}(A,B):=\max\Bigl\{\,\sup_{x\in A}\inf_{y\in B}\Vert x-y\Vert,\,\,\sup_{y\in B}\inf_{x\in A}\Vert x-y\Vert\Bigr\}.
\end{equation}
Dobrushin, Koteck\'y and Shlosman~\cite[Theorem~2.3]{Book:DKS} expressed the uniqueness of the minimizer quantitatively as follows: For any rectifiable Jordan curve~$\lambda$ enclosing a region of a unit Lebesgue area,
\begin{equation}
\label{E:1.15a}
\inf_{\xi\in\R^2}\, d_{\rm H}\bigl(\xi+\intr(\lambda),\widehat{W}_p\bigr)\le C_p\frac{\,\sqrt{\,\len_{\beta_p}(\lambda)^2-\len_{\beta_p}(\hat{\gamma}_p)^2}}{\len_{\beta_p}(\hat{\gamma}_p)^2}.
\end{equation}
Here $C_p$ is a constant depending 
only on~$\beta_p$. In~\cite{Book:DKS}, the bound is on 
$d_{\rm H}\bigl(\xi+\lambda,\hat{\gamma}_p\bigr)$, but as the interior of both curves has the same Lebesgue measure, it readily extends to $d_{\rm H}\bigl(\xi+\intr(\lambda),\widehat{W}_p\bigr)$. The proof in~\cite{Book:DKS} uses a generalized Bonnesen inequality for the metric on~$\R^2$ induced by~$\beta_p$ (see \cite[Section~2.5]{Book:DKS} or the monograph by Burago and Zalgaller~\cite[Theorem~1.3.1]{Burago-Zalgaller}).

\smallskip
As a consequence, 
we are able to derive an almost sure shape theorem for the isoperimetric sets of $\Phi_{\mbC^\infty, 0}$ and $\wt{\Phi}_{\mbC^n}$. More explicitly, conditioned on $0 \in \mbC^\infty$, let~$\hat{\UU}_{\mbC^\infty}(r)$ denote the set of minimizers for \eqref{thm:Profile} with $G:=\mbC^\infty$.
Similarly, write~$\hat{\UU}_{\mbC^n}$ for the set of minimizers for \eqref{eqn:Cheeger'} with $G:=\mbC^n$ and $H:=\mbC^\infty$. We have:

\begin{theorem}[Limit shape --- isoperimetric profile]
\label{thm-1.6}
Let $d=2$, $p > p_\cc(\Z^2)$ and let $\widehat W_p$ be related to $\beta_p$ 
as in \eqref{E:1.12x}. Then 
\begin{equation}
\label{E:1.16c}
    \max_{U \in \hat{\UU}_{\mbC^\infty}(n)}
       \,\, \inf_{\xi \in \R^2}
            d_{\rm H} \bigl(n^{-1/2} U , \,\,  \xi + \theta_p^{-1/2}\widehat{W}_p \bigr)
                \,\,\underset{n\to\infty}\longrightarrow\,\, 0
\end{equation}
and
\begin{equation}
\label{E:1.15}
    \max_{U \in \hat{\UU}_{\mbC^\infty}(n)} \big| |U|/n - 1 \big| \,\,\underset{n\to\infty}\longrightarrow\,\, 0
\end{equation}
hold for $\BbbP(\cdot | 0 \in \mbC^\infty)$-almost every realization of~$\omega$.
\end{theorem}

For the minimizers of \eqref{eqn:Cheeger'} we similarly get:

\begin{theorem}[Limit shape --- Cheeger constant]
\label{thm-1.7}
Let $d=2$, $p > p_\cc(\Z^2)$ and let $\widehat W_p$ be related to~$\beta_p$ 
as in \eqref{E:1.12x}. Then, for~$n$ sufficiently large, all minimizers are connected (as induced subgraphs of~$\mbC^n$) and
\begin{equation}
\label{E:1.17c}
    \max_{U \in \hat{\UU}_{\mbC^n}}
        \inf_{\xi \in \R^2}
            d_{\rm H}\bigl( n^{-1}U , \,\,  \xi + \sqrt{2}\,\wh{W}_p \bigr) \,\,\underset{n\to\infty}\longrightarrow\,\, 0
\end{equation}
and
\begin{equation}
\label{E:1.18}
    \max_{U \in \hat{\UU}_{\mbC^n}} \left| \frac{|U|}{(\theta_p |\rmB_\infty(n)|/2)} - 1 \right| \,\,\underset{n\to\infty}\longrightarrow\,\, 0
\end{equation}
hold for $\BbbP$-almost every realization of~$\omega$.
\end{theorem}

\subsection{Discussion and open problems}
We finish the introduction with a brief discussion of various limitations of our results; we use 
this as an opportunity to point out some open problems.

\podsekce{Free and periodic boundary conditions}
As noted before Theorem~\ref{thm:Cheeger}, our results on the Cheeger constant are limited to the quantity~$\widetilde\Phi_{\mbC^n}$, which includes the edges ``sticking out'' of~$\rmB_\infty(n)$. Conjecture~\ref{con1} is instead formulated for ``free'' boundary conditions (no edges sticking out), but one can also take periodic boundary conditions (edges that ``stick out'' connect to vertices on the opposite side of the box). A question is 
whether these cases can be resolved as well. The reduction to an isoperimetric problem in~$\R^2$ for the norm~$\beta_p$ should still be feasible; a difficult part is the analysis of the minimizers for the continuum isoperimetric problem. For $p=1$, the solution is half of the square (free b.c.) or a band around the torus (periodic b.c.). However, it is not at all clear how this changes once~$p$ is  lowered significantly below~$1$.

\podsekce{Regularity of minimizers}
Related to the previous problem is the question of regularity of the norm~$\beta_p$ (beyond its continuity which is automatic) and thus also the regularity of the limiting curve. In particular it is not clear whether $\hat{\gamma}_p$ may have cusps or flat portions.

\podsekce{Near criticality}
The limit shape is defined for all $p>p_\cc$, but not for $p_\cc$, for which there is no percolation. Notwithstanding, one is naturally interested in the behavior of the shape in the limit when $p\downarrow p_\cc$. In analogy to the Ising model and super-critical percolation, we conjecture that the limit is the Euclidean 
ball.

\podsekce{Near $p=1$}
On the other hand, it is easy to show that when $p \uparrow 1$, the boundary norm $\beta_p$ converges to the $\ell^1$-norm uniformly on $\{x \in \R^2 \colon  \|x\|_2 = 1\}$. Consequently, $\widehat W_p$ converges to the 
box of unit volume (in the Hausdorff metric on $\R^2$).

\podsekce{Size of the holes}
Although the definition of the Cheeger constant ensures that a minimizer can always be taken connected with a connected complement, neither our modified definition nor that for the isoperimetric profile ensures the connectedness of the complement and, in fact, they need not be such. Our proof gives an approximation, to the leading order in linear size, of the isoperimetric sets by a convex set in the continuum. The estimates on the Hausdorff distance then show that the diameter of the potential ``holes'' is negligible compared to the diameter of the whole set. We believe that they should not be significantly larger than the log of the diameter.

\podsekce{Deviation tails}
Another question of reasonable interest is whether one can derive (reasonably) sharp probabilistic estimates for finding isoperimetric sets of a given (large) volume whose boundary length or shape significantly deviates from the limiting values. We in fact tend to expect these tail-estimates to exhibit different scaling for positive and negative deviations; similarly to those found for passage times in first passage percolation (see Kesten~\cite{Kesten-first-passage}).

\podsekce{Shape fluctuations}
Related to this is the question of shape fluctuations. These have been addressed in the context of ``classical'' Wulff constructions by Alexander~\cite{Alexander}, Hammond and Peres~\cite{Hammond-Peres} and Hammond~\cite{Hammond-1,Hammond-2,Hammond-3}. An interesting feature there is that the fluctuations are of order~$n^{1/3}$. It is interesting to ponder about the connection to (still conjectural) $n^{2/3}$-scaling of transversal fluctuations for minimal-length paths in first passage percolation; cf~Chatterjee~\cite{Chatterjee}, Auffinger and Damron~\cite{Auffinger-Damron} for recent work on this. The problem of fluctuations of minimal-length paths can be formulated in the present context as well, see Section~\ref{sec2}.

\podsekce{Higher dimensions}
A final, and at this point completely open, problem is that of dimensions $d\ge3$. In the context of Wulff construction for percolation and Ising model, generalizing the two-dimensional proofs to higher dimensions required very considerable effort (Cerf~\cite{Cerf}, Bodineau~\cite{Bodineau}, see the reviews by Cerf~\cite{Cerf-book} and Bodineau, Ioffe and Velenik~\cite{BIV}). Although some inspiration can be drawn from these for our problem as well, it will probably be limited to facts about geometry of random surfaces and isoperimetric problems in~$\R^d$.

\subsection{Outline}
The remainder of this paper is organized as follows: In Section~\ref{sec2} we introduce the notions of a right-most path and its right-boundary length and use these to define the boundary norm~$\beta_p$. In Section~\ref{sec3} we then derive concentration estimates that permit us to control the rate of convergence of the right-boundary length (scaled by the geometric distance) to~$\beta_p$. Section~\ref{sec4} then shows that circuits on the lattice can be matched with closed curves in $\R^2$, such that the right-boundary length of the former is approximately the $\beta_p$-length of the latter. This, in turn, is used in Section~\ref{sec5} to control the size of the boundary of subsets of $\mbC^\infty$ or~$\mbC^n$ and thereby prove the main theorems.

\subsection{General notation}
In general, $\|x\|_q$ will denote the $\ell^q$-norm of $x \in \R^2$. However, we will regularly write $\|x\|$ for~$\|x\|_\infty$. We set $\rmB_q(r) := \{x \in \R^2 \colon  \|x\|_q \leq r\}$ but, to simplify notation, we will often regard $\rmB_\infty(n)$ also as the $\ell^\infty$-ball in~$\Z^2$. The length of a curve $\lambda$ with respect to the $\Vert\cdot\Vert_q$-norm will be denoted by $\len_q(\lambda)$. The notation $\poly(x,y)$ stands for the closed linear segment in~$\R^2$ from~$x$ to~$y$. A polygonal line is then the curve $\poly(x_0,\dots,x_n):=\poly(x_0,x_1)\circ\dots\circ\poly(x_{n-1},x_n)$ where ``$\circ$'' denotes the usual concatenation of curves. Finally, we shall write $C$,$C'$, $C''$, etc. to denote non-negative constants whose value may change line by line, unless stated otherwise.

\section{The Boundary Norm}
\label{sec2}\noindent
In this section we define the boundary norm $\beta_p$ that lies at the core of all of our arguments. A key notion is that of a right-most path which, as we shall see later, will be used to characterize the ``shape'' of finite sets in $\Z^2$.

\subsection{Right-most paths and their right boundaries}
Consider an unoriented planar graph~$G=(V,E)$ embedded in $\R^2$. A path $\gamma$ in~$G$ from~$x$ to~$y$ of length $|\gamma| = n$, is a sequence of vertices $(x_0,x_1,\dots,x_n)$ such that $(x_i,x_{i+1})\in E$ for every $i\in\{0,\dots,n-1\}$ and $x_0=x$, $x_n=y$. The path is called \emph{simple} if it traverses each edge in~$G$ at most once in each direction (that is $\gamma$ is simple in the directed version of $G$). The path is called a \emph{circuit} if $x_n = x_0$; in this case we may identify indices modulo~$n$ (e.g. $x_{n+1} = x_1$). 

When both~$x_{i-1}$ and~$x_{i+1}$ are defined, the {\em right-boundary edges} at vertex $x_i$ are obtained by listing all (oriented) edges emanating from~$x_i$ counter-clockwise starting from, but not including, $(x_i,x_{i-1})$ and ending with, but not including, $(x_i,x_{i+1})$. If~$x_{i-1}$ or~$x_{i+1}$ is undefined, which can only happen at the endpoints of a non-circular path, the set of boundary edges at~$x_i$ is empty. The \emph{right boundary} $\partial^+\gamma$ of~$\gamma$ is the set of all right-boundary edges at all vertices of $\gamma$, see~Fig.~\ref{fig1}. Notice that if $\gamma$ visits a vertex multiple times, each visit may contribute distinct right-boundary edges to~$\partial^+\gamma$. A path is said to be \emph{right-most} if it is simple and it does not use any edge (regardless of orientation) in its right boundary. Let $\RR(x,y)$ be the set of all right-most paths from~$x$ to~$y$.

\newcounter{obrazek}

\begin{figure}[t]

\refstepcounter{obrazek}
\centerline{\includegraphics[width=4.7in]{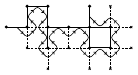}}
\begin{quote}
\small
{\sc Fig.~\theobrazek\ }
\label{fig1}
An example of a right-most path $\gamma$ (solid edges) and its right-boundary~$\partial^+\gamma$ (dashed edges). The associated interface (wiggly curve on the medial graph) reflects on the edges in $\gamma$ and cuts through the edges in~$\partial^+\gamma$.
\normalsize
\end{quote}
\end{figure}

Now fix $G:=\Z^2$. For $\omega \in \Omega$ and a right-most path $\gamma$, set
\begin{equation}\label{eqn:1.3}
    \frb(\gamma) := 
    \bigl|\{ e \in \partial^+ \gamma \colon  \omega(e) = 1\}\bigr|.
\end{equation}
This is the \emph{right-boundary length} of~$\gamma$ in configuration $\omega$. Note that $\partial^+\gamma$ may include an edge in both orientations; both of these then contribute to~$\frb(\gamma)$.
If $x$ and $y$ are connected in $\omega$, we then define the \textit{right-boundary distance} by
\begin{equation}\label{eqn:903.2}
    b(x, y) :=\inf\bigl \{\frb(\gamma) \colon  \gamma \in \RR(x,y),\, \text{open}\bigr\}.
\end{equation}
Finally, to ensure containment of the arguments of $b(\cdot,\cdot)$ in $\mbC^\infty$, for each~$x\in\R^d$ we define the ``nearest vertex'' $[x]\in\mbC^\infty$ as follows: Suppose the probability space is large enough to carry a collection of random variables $\{\eta_z\}_{z\in\Z^2}$ that are i.i.d.\ uniform on~$[0,1]$ and independent of~$\omega$ under~$\BbbP$. For $x \in \R^2$, we then let $[x]$ denote the vertex $z$ on $\mbC^\infty$ which is nearest to $x$ in the $\ell^\infty$-norm, taking the one with a minimal~$\eta_z$ in case there is a tie. Obviously, $[x]$ depends on~$\omega$ and the $\eta$'s but we will not make this notationally explicit.

\smallskip
The main result of Section~\ref{sec2} is:

\begin{theorem}[The boundary norm]
\label{thm:BetaP}
For any $p>p_\cc(\Z^2)$ and any $x \in \R^2$, the limit
\begin{equation}
\label{eqn:902}
	\beta_p(x) 
	:= \lim_{n \to \infty} \frac{b([0], [nx])}{n}
\end{equation}
exists $\BbbP$-a.s.\ and is non-random, non-zero (for $x\ne0$) and finite. The limit also exists in $L^1$ and the convergence is uniform on $\{x\in\R^2\colon\: \|x\|_2 = 1\}$. Moreover, 
\begin{enumerate}
\item[(1)]
$\beta_p$ is homogeneous, i.e., $\beta_p(\lambda x)=|\lambda|\beta_p(x)$ for all $x\in\R^2$ and all $\lambda \in \R$, and
\item[(2)]
$\beta_p$ obeys the triangle inequality,
\begin{equation}
\beta_p(x+y)\le\beta_p(x)+\beta_p(y),\qquad x,y\in\R^2.
\end{equation}
\end{enumerate}
In other words, $\beta_p$ is a norm on $\R^2$.
\end{theorem}

\noindent The norm $\beta_p$ inherits all symmetries of the lattice as the following proposition shows.
\begin{proposition}
\label{prop:SymmetryForBeta}
Let $\beta_p$ be defined as in Theorem~\ref{thm:BetaP}. Then for all $(x_1, x_2) \in \R^2$,
\begin{equation}
	\beta_p\bigl((x_1, x_2)\bigr)=\beta_p\bigl((x_2, x_1)\bigr)=\beta_p\bigl((\pm x_1,\pm x_2)\bigr)
\end{equation}
for any choice of the two signs~$\pm$. 
\end{proposition}

The remainder of Section~\ref{sec2} is devoted to the proof of Theorem~\ref{thm:BetaP} and Proposition~\ref{prop:SymmetryForBeta}. As a preparation, we will first need to introduce some geometric facts concerning right-most paths and their right boundaries and also some relevant properties of bond percolation on~$\Z^2$.

\begin{figure}[t]
\refstepcounter{obrazek}
\centerline{\includegraphics[width=3.8in]{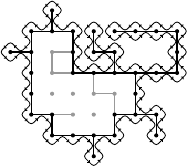}}
\bigskip
\begin{quote}
\small
{\sc Fig.~\theobrazek\ }
\label{fig1a}
The boundary interface $\partial$ (wiggly line) winding around a finite connected subgraph of~$\Z^2$ (solid edges). The black edges are those lying on the right-most circuit~$\gamma$ associated with~$\partial$. The vertices enclosed by $\partial$ (gray and black bullets) are part of the set $\vol(\gamma)$ to be defined and used later.
\normalsize
\end{quote}
\end{figure}

\subsection{Geometry of right-most paths}
Our arguments will rely heavily on planar duality. Recall that with each planar graph $G$ we can associate another planar graph, its \emph{dual}~$G_\star$, by identifying the faces of~$G$ with the vertices of~$G_\star$ and then, naturally, a \emph{dual edge}~$e_\star$ with each \emph{primal edge}~$e$. If~$e$ is oriented, we orient~$e_\star$ so that it points from the face on the left of $e$ to the face on the right of~$e$. (The dual of~$e_\star$ is then the reversal of~$e$.) 

With a planar graph $G=(V,E)$ we can also associate the so-called \emph{medial graph}, $G_\sharp=
(V_\sharp,E_\sharp)$. This is the graph with vertices~$V_\sharp:=E$ and an edge $(e,e')\in E_\sharp$ whenever $e$ and $e'$ are adjacent edges on a face of~$G$. We orient the edges in~$E_\sharp$ clockwise in each face of~$G$; thanks to planarity, edges will then be oriented counterclockwise around each vertex of $G$ (or, equivalently, in each face of $G_\star$). An \emph{interface} is then an edge self-avoiding (oriented) path in the graph $G_\sharp$ which does not use its initial or terminal vertex more than once, except to close a cycle. The medial graph of~$G$ is also the medial graph of~$G_\star$, once we reverse the orientation of its edges. Therefore, an interface in~$G_\sharp$ run backwards is an interface in $(G_\star)_\sharp$.

As Fig.~\ref{fig1} shows, when an interface $\partial$ visits an edge of~$G$ it either ``reflects'' on it or ``cuts-through'' it (reflecting on its dual). More precisely, let $e_0, \dots, e_n$ be the sequence of edges of~$G$ (i.e., vertices of the medial graph) visited by~$\partial$. Whenever $e_{i-1},e_i,e_{i+1}$ are well defined, we say that the interface \emph{reflects on~$e_i$} if $e_{i-1},e_i,e_{i+1}$ lie on the same face of~$G$; otherwise it \emph{cuts through~$e_i$}. Notice that unless $\partial$ is a cycle, it does not reflect or cut through~$e_0$, $e_n$.

For any connected finite subgraph $H \subset G$, there is a unique interface $\partial$ in $G_\sharp$ which surrounds~$H$ and reflects only on edges which belong to $H$ (see Fig~\ref{fig1a}). We shall call $\partial$ the \emph{outer boundary interface} of~$H$. Necessarily, $\partial$ goes around~$H$ in the counterclockwise direction.
Associated with $\partial$ there are two circuits, $\gamma$ in~$G$ and $\gamma_\star$ in~$G_\star$ that describe equally well the outer ``shape'' of~$H$. Here~$\gamma$ is obtained by traversing $\partial$ and listing all edges of~$G$ on which the interface reflects while~$\gamma_\star$ is obtained similarly from the reversal~$\partial_\star$ of~$\partial$ viewed as an interface on~$G_\star$. A key fact for our purposes is that both $\gamma$ and~$\gamma_\star$ are right-most
and that the edges in $\partial^+ \gamma$ are (outer) boundary edges of $H$ (i.e. edges which emanate from vertices in $H$ but not included in its edge-set):
\begin{proposition}
\label{prop-duality}
For each outer boundary interface~$\partial = (e_1,\dots,e_m)$, the sub-sequence of edges $(e_{k_1},\dots,e_{k_n})$ that are not cut through by $\partial$ form a right-most path $\gamma$. This mapping is one-to-one and onto the set of all right-most paths. In particular, $\gamma$ is a right-most circuit if and only if $\partial$ is a cycle in the medial graph. Finally, the edges in $\partial \setminus (e_{k_1},\dots,e_{k_n})$ (oriented properly) form $\partial^+\gamma$.
\end{proposition}

\begin{proofsect}{Proof}
Consider an interface $\partial:=(e_1,\dots,e_m)$. If~$\partial$ reflects on~$e_i$, orient~$e_i$ in the ``direction''~$\partial$ sweeps by it, if $\partial$ cuts through~$e_i$, orient~$e_i$ so that it points left-to-right as it is traversed by~$\partial$. Let $(e_{k_1},\dots,e_{k_n})$ be as above. It is easy to check that, for each $i=1,\dots,n-1$, the initial vertex of each edge among $e_{k_i+1},\dots,e_{k_{i+1}}$ is the terminal vertex of~$e_{k_i}$. Calling this vertex~$x_i$, we obviously have~$e_{k_i}=(x_{i-1},x_i)$ and so, setting $x_0$ to the initial vertex of~$e_{k_1}$ and $x_n$ to the terminal vertex of~$e_{k_n}$, the sequence~$\gamma:=(x_0,\dots,x_n)$ is a path. The orientation of~$G_\sharp$ and the fact that~$\partial$ is edge-simple ensures that~$\gamma$ is simple. Moreover, the edges $e_{k_{i}+1},\dots,e_{k_{i+1}-1}$ are then the right-boundary edges of~$\gamma$ at~$x_i$. Finally, $\gamma$ is right-most because once $\partial$ cuts through an edge (making it part of $\partial^+\gamma$) it cannot reflect on it later.

For the opposite direction, let $\gamma:=(x_0,\dots,x_n)$ be a right-most path. The corresponding interface~$\partial = (e_1, \dots, e_m)$ is constructed as follows: The first edge in $\partial$ will be $(x_0,x_1)$. Then, for each $k=1, \dots, N$ such that both $(x_{k-1}, x_k)$ and $(x_k, x_{k+1})$ are on $\gamma$, we add to $\partial$ all right-boundary edges emanating from $x_k$ in a counter-clockwise order and then also $(x_k, x_{k+1})$.

The construction ensures that~$\partial$ is oriented in accord with~$G_\sharp$. To see that $\partial$ is edge-simple, suppose the opposite. Then $\partial$ uses a pair $(e_i,e_{i+1})$ one more time at a later index. If $\partial$ reflects on $e_{i+1}$, then $e_{i+1}$ is an edge in~$\gamma$ and so it cannot be seen (and reflected upon) by~$\partial$ again as~$\gamma$ is itself simple. If, on the other hand, $\partial$ cuts through~$e_{i+1}$, then $e_{i+1}$ is encountered twice as a right-boundary edge of a vertex in $\gamma$, which again cannot happen since $\gamma$ is edge simple. Thus~$\partial$ is edge simple. Similar considerations show that the first and last edge cannot repeat unless $\gamma$ is a circuit. This shows that $\partial$ is an interface. Since $\partial$ uses all edges of $\gamma$ without cutting through them, applying the mapping in the statement to $\partial$ will give back $\gamma$. This shows that the map $\partial\mapsto\gamma$ is onto. An inspection of the above shows that same applies to $\gamma\mapsto\partial$.
\end{proofsect}


\begin{corollary}
\label{cor-duality}
For any right-most path~$\gamma$, the set of edges dual to $\partial^+\gamma$ defines a right-most path~$\gamma_\star$ on~$G_\star$. 
\end{corollary}

\begin{proofsect}{Proof}
Consider the interface $\partial$ related to~$\gamma$ as stated in Proposition~\ref{prop-duality}. Then its reversal $\partial_\star$ is an interface on~$G_\star$ and so it induces a right-most path $\gamma_\star$. Obviously, $e_\star$ is an (oriented) edge in $\gamma_\star$ if and only if its primal edge~$e$ belongs to~$\partial^+\gamma$.
\end{proofsect}

The graph $G:=\Z^2$ is certainly planar and it will have both a dual and a medial. Using the standard embedding of~$\Z^2$ into~$\R^2$, the dual of~$\Z^2$ can be identified with $\Z^2_\star=(\ffrac12,\ffrac12)+\Z^2$ and its medial~$\Z^2_\sharp$ with a scaled and rotated copy of~$\Z^2$. Proposition~\ref{prop-duality} explains our reliance on right-most paths in the definition of~$\beta_p$. We finish with two simple lemmas that will be needed for the proof of Theorem~\ref{thm:BetaP}. 

\begin{lemma}
\label{prop:DualPathLength}
For every right-most path $\gamma$,    
\begin{equation}
    \label{eqn:A2.8}
        \frac{|\gamma|}{3} - 2 \leq \bigl| \partial^+\gamma \bigr| \leq 3|\gamma| \, .
    \end{equation}
\end{lemma}

\begin{proofsect}{Proof}
Let $\gamma$ be a right-most path and $\partial$ its associated interface. Since the degree of each vertex in~$\Z^2$ is four, in every four steps, $\partial$ visits at least one edge of~$\gamma$ and at least one edge of~$\partial^+\gamma$; otherwise, it would not be edge-simple. Since no (oriented) edge is visited more than once, the claim follows.
\end{proofsect}

We shall now define a way to concatenate two adjacent right-most paths such that the resulting path is also right-most. Let $\gamma = (u_0, u_1, \dots, u_n)$ and $\gamma' = (v_0, v_1, \dots, v_{m})$ be paths and set
\begin{equation}
	k := \min \{i \colon  u_i \in \gamma'\} \quad\text{and}\quad
	l := \max \{i \colon  v_i = u_k \} .
\end{equation}
Then the $*$-concatenation of $\gamma$ and $\gamma'$ denoted $\gamma * \gamma'$ is the path
$(u_0, \dots, u_k, v_{l+1}, \dots, v_m)$. 

\begin{lemma}
\label{prop:MergingOfPaths}
For any $\gamma\in\RR(x,y)$ and $\gamma'\in\RR(y,z)$ we have $\gamma'' := \gamma * \gamma'\in\RR(x,z)$. Moreover,\begin{equation}
\label{eqn:903}
	\big| (\partial^+\gamma'') \setminus 
		( \partial^+\gamma \cup 
		\partial^+\gamma' ) \big| \leq 2.
\end{equation}
\end{lemma}

\begin{proof}
The case $x=z$ is trivial, since in this case $\gamma'' = (x)$. Otherwise, clearly $\gamma''$ is a path from $x$ to $z$. Then let
\begin{equation}
	\gamma_L := (u_0, \dots, u_k), \quad
	\gamma_M := (u_{k-1}, u_k=v_l, v_{l+1}) \quad\text{and}\quad
	\gamma_R := (v_{l}, \dots, v_m) .
\end{equation}
Note these are all right-most paths. From the construction it follows that $\partial^+ \gamma_L$ has no edges in common with $\gamma_M \cup \gamma_R$, and $\partial^+ \gamma_R$ has no edges in common with $\gamma_L \cup \gamma_M$ and that $\partial^+ \gamma_M$ has no edges in common with $\gamma_L \cup \gamma_R$. Hence, $\gamma''$ is right-most. Moreover, $\partial^+ \gamma_L \subseteq \partial^+ \gamma$, $\partial^+ \gamma_R \subseteq \partial^+ \gamma'$ and $|\partial^+ \gamma_M| \leq 2$. Since also
$\partial^+ \gamma'' = \partial^+ \gamma_L \cup \partial^+ \gamma_M \cup \partial^+ \gamma_R$, the claim follows. 
\end{proof}

\subsection{Percolation inputs}
In this section we will assemble some useful facts concerning percolation on~$\Z^2$. Percolation on~$\Z^d$ is a well studied subject; see Grimmett~\cite{grimmett1999percolation} for a standard reference. For our purposes, we will need the following three facts:
\settowidth{\leftmargini}{(11)}
\begin{enumerate}
\item[(1)] \emph{Exponential decay of subcritical connectivities}: Let $\mbC(0)$ denote the connected component of~$0$ in~$\omega$. For each $p<p_\cc(\Z^2)$, there are $C,C'<\infty$
\begin{equation}
\label{E:2.10b}
\BbbP\bigl(|\mbC(0)|\ge n)\le C\texte^{-C n},\qquad n\ge1.
\end{equation}
In particular, if $x\leftrightarrow y$ denote the event that $x$ and~$y$ are connected by an open path in~$\omega$, then
\begin{equation}
\label{E:2.9a}
\BbbP(x\leftrightarrow y)\le C\texte^{-C'\Vert x-y\Vert},\qquad x,y\in\Z^2.
\end{equation}
\item[(2)] \emph{Duality}: The planar nature of~$\Z^2$ permits us to encode a percolation configuration $\omega$ by means of its dual counterpart $\omega_\star$ which is defined by
\begin{equation}
\omega_\star(e_\star):=1-\omega(e).
\end{equation}
Note that $\BbbP(\omega_\star(e_\star)=1)=1-p$, so the dual edges are occupied with dual probability $p_\star:=1-p$. Thanks to Kesten's celebrated result~\cite{Kesten} we know that $p_\cc(\Z^2)=\ffrac12$ and so $\BbbP(\omega \in \cdot)$ is supercritical ($p > p_\cc(\Z^2)$) if and only if $\BbbP(\omega_* \in \cdot)$ is a subcritical ($p_\star<p_\cc(\Z^2)$). In particular, \eqref{E:2.9a} applies to $\omega_\star$ for all $p>p_\cc(\Z^2)$.
\item[(3)] \emph{Comparison of graph and lattice distance}:
Whenever $x\leftrightarrow y$, we can define $D_\omega(x,y)$ as the length of the shortest open path connecting~$x$ to~$y$. (When $x\nleftrightarrow y$ we set $D_\omega(x,y)=\infty$.) Thanks to a result of Antal and Pisztora~\cite[Theorem~1.1]{Antal-Pisztora} we know that, for any $p>p_\cc(\Z^2)$ there is $\rho=\rho(p,d)$ such that
\begin{equation}
\label{E:2.10a}
\limsup_{\|y\|\to\infty}\,\frac1{\|y\|}\,\log
\BbbP\bigl(0\leftrightarrow y,\,D_\omega(0,y)>\rho\|y\|\bigr)<0.
\end{equation}
In particular, $D_\omega(x,y)$ is at large distances comparable with the lattice distance.
\end{enumerate}
We will now use these facts to prove that the distance between the infinite cluster and any fixed point on the lattice has exponential tails.

\begin{lemma}
\label{prop:DistanceToCInf}
Suppose $p>p_\cc(\Z^2)$. There are $C,C' > 0$ such that for all $x \in \Z^2$ and $r > 0$,
\begin{equation}
\label{E:2.7c}
	\BbbP\bigl(\|[x] - x\|) > r\bigr) \leq C\texte^{-C' r}
\end{equation}
\end{lemma}

\begin{proof}
On $\{\|[x] - x\|>r\}$, where $r\in\N$, there is no point of~$\mbC^\infty$ in the box of side~$2r+1$ centered at~$x$. By duality, this box is therefore circumnavigated by a dual path whose edges are open (in~$\omega_\star$). In particular, there are vertices $x_+,x_-\in\Z^2$ of the form $x_\pm:=x\pm n_\pm e_1$, where $e_1$ is the unit vector in the first coordinate direction and $n_\pm\ge r$, whose dual neighbors lie on this path. Since $p>p_\cc(\Z^2)$, the dual percolation is subcritical and the probability that this occurs is exponentially small in $n_++n_-$. Summing over $n_\pm\ge r$, we get \eqref{E:2.7c}.
\end{proof}

The following lemma is an extension of \eqref{E:2.10a}, of which we will make frequent use.
\begin{lemma}
\label{prop:ImprovedPisztora}
Suppose $p>p_\cc(\Z^2)$. There are $\alpha_1,C,C' > 0$
such that for all $x,y \in \Z^2$,
\begin{equation}
\BbbP\bigl(D_\omega([x], [y]) > r\bigr)\le C\texte^{-C' r},\qquad r > \alpha_1 \|y-x\|.
\end{equation}
\end{lemma}

\begin{proofsect}{Proof}
In light of Lemma~\ref{prop:DistanceToCInf} and the translation invariance of~$\BbbP$, it suffices to show that, for some $\alpha,C,C'\in(0,\infty)$ and any $y\in\Z^2$,
\begin{equation}
\label{E:2.9}
\BbbP\bigl(\,0,y\in\mbC^\infty,\,\,D_\omega(0,y) >\alpha r\bigr) \leq  C\texte^{-C' r},\qquad r>\|y\|.
\end{equation}
We will invoke \eqref{E:2.10a} and a short argument. Recall that $\rmB_\infty(r):=\{y\in\Z^2\colon\|y\|\le r\}$. Thanks to the triangle inequality for $D_\omega$, on $\{0,y\in\mbC^\infty,\,\,D_\omega(0,y) >5\rho r\}$, where~$\rho$ is as in \eqref{E:2.10a}, there must be a point $z\in\partial\rmB_\infty(2r)\cap\mbC^\infty$ such that $\{0\leftrightarrow z,\,D_\omega(0,z)>2\rho r\}\cup\{y\leftrightarrow z,\,D_\omega(y,z)>3\rho r\}$ holds. Hence,
\begin{multline}
\qquad
\BbbP\bigl(0,y\in\mbC^\infty,\,\,D_\omega(0,y) >5\rho r\bigr)
\\
\le\sum_{z\in\partial\rmB_\infty(2r)}\Bigl(\BbbP\bigl(0\leftrightarrow z,\,D_\omega(0,z)>2\rho r\bigr)+\BbbP\bigl(y\leftrightarrow z,\,D_\omega(y,z)>3\rho r\bigr)\Bigr).
\qquad
\end{multline}
Assuming $\Vert y\Vert\le r$, \eqref{E:2.10a} along with $r\le\Vert z\Vert\le 2r$ and $r\le\Vert y-z\Vert\le 3r$ imply that both probabilities on the right are bounded by $C\texte^{-C'r}$ for some $C,C'>0$, independent of~$y$ and~$z$. As $|\partial \rmB_\infty(2r)|=O(r)$, the bound \eqref{E:2.9} follows with $\alpha:=5\rho$.
\end{proofsect}

Finally we prove that a right-most path $\gamma$ cannot have too few open edges in $\partial^+\gamma$.

\begin{proposition}
\label{prop:RenormalizedLD}
For $p>p_\cc(\Z^2)$ there are $\alpha_2,C,C' > 0$, such that for  all $n \geq 0$,
\begin{equation}
\label{eqn:931}
	\BbbP\Bigl(\exists \gamma \in \bigcup_{x \in \Z^2} \RR(0,x) \colon  |\gamma| \geq n 
		, |\frb(\gamma)| \leq \alpha_2 n\Bigr) \leq C\texte^{-C' n} .
\end{equation}
\end{proposition}

\begin{proof}
Recall that $\omega_\star$ is the dual configuration of $\omega$, that is
$\omega_\star(e_\star) = 1-\omega(e)$. By Corollary~\ref{cor-duality} and Lemma~\ref{prop:DualPathLength}, on the event in \eqref{eqn:931} there is a dual right-most path and, in particular, a simple path $\gamma_\star$ which ends at a dual neighbor of~$0$, has length at least $n/3 - 2$ and contains less than $\alpha_2 n$ edges~$e$ with $\omega_\star(e)=0$. 
As there are only four dual neighbors of~$0$, replacing $\omega_\star$ by~$\omega$ for ease of notation, it suffices to show 
\begin{equation}
\label{E:2.10}
    \BbbP_{p_\star} \bigl(\exists \gamma\colon 0\in\gamma \text{, simple},\, 
        |\gamma| = n ,\, |\{e \in \gamma \colon  e \notin \omega\}| \leq \alpha_2 n\bigr) \leq C \texte^{-C' n} ,
\end{equation}
where $p_\star:=1-p$ is the dual to~$p$ and $\BbbP_{p_\star}$ is the dual percolation measure.
    
To prove \eqref{E:2.10}, let us define, for any $\alpha \geq 0$, the set
\begin{equation}
    \AA_{n, \alpha} := \big \{
        \exists \gamma\colon 0\in\gamma \text{, simple},
		|\gamma| = n ,\, |\{e \in \gamma \colon  e \notin \omega\}| \leq 
			\alpha n \big \}.
\end{equation}
Note that $\AA_{n,\alpha}$ depends only on the edges in a box of side length~$n$ centered at the origin. This permits us to regard $\AA_{n,\alpha}$ as a subset of a finite sample space. 

Fixing an (arbitrary) ordering of $\{\gamma\colon 0\in\gamma \text{, simple},\, |\gamma|=n\}$, define the map $\TT_{n, \alpha} \colon  \AA_{n, 0} \to 2^{\AA_{n, \alpha}}$ as follows: Given $\sigma \in \AA_{n, 0}$, let $\gamma$ be the simple $\sigma$-open path of length $n$ from the origin that is minimal in the above ordering. Then $\TT_{n, \alpha}(\sigma)$ is the set of all configurations that are obtained from~$\sigma$ by closing at most $\alpha n$ edges in $\gamma$. We claim
\begin{equation}
\label{E:2.15d}
    \AA_{n, \alpha} \subseteq \bigcup_{\sigma \in \AA_{n,0}} \TT_{n, \alpha}(\sigma).
\end{equation}
Indeed, 
given $\sigma' \in \AA_{n, \alpha}$, find the minimal (in the above ordering) simple path $\gamma$ which has at most $\alpha n$ closed edges in~$\sigma'$. Then set $\sigma$ to be the configuration obtained from $\sigma'$ by opening all edges along this path. Clearly, $\sigma \in \AA_{n, 0}$ and $\sigma' \in \TT_{n, \alpha}(\sigma)$, proving \eqref{E:2.15d}.

For brevity let $\BbbP_{p_\star}(\sigma)$ denote the probability of the configuration~$\sigma$. (Recall that we regard $\sigma$ as a configuration on a finite set of edges only.) If $\sigma'\in\TT_{n,\alpha}(\sigma)$ has~$k$ close edges on the minimal $\sigma$-open path~$\gamma$, then $\BbbP_{p_\star}(\sigma')=(\frac p{1-p})^k\BbbP_{p_\star}(\sigma)$. The union bound then shows
\begin{equation}
    \BbbP_{p_\star}(\TT_{n, \alpha}(\sigma)) \leq
       \biggl[ \sum_{k=0}^{\lfloor\alpha n\rfloor} \left( \frac p{1-p} \right)^{k}\binom nk\biggr]      
       \BbbP_{p_\star}(\sigma).
\end{equation}
Denoting
\begin{equation}
c(\alpha, p) := \sup_{0\le s\le\alpha}\biggl\{
s\log\frac p{1-p} - \bigl(s\log s + (1-s) \log (1-s)\bigr)\biggr\}
\end{equation}
the Stirling bound gives us
\begin{equation}
\BbbP_{p_\star}(\TT_{n, \alpha}(\sigma))\le Cn\,\texte^{c(\alpha,p)n}\,\BbbP_{p_\star}(\sigma)
\end{equation}
and summing over $\sigma$ yields
\begin{equation}
    \BbbP_{p_\star}(\AA_{n, \alpha}) \leq \sum_{\sigma \in \AA_{n, 0}}
        \BbbP_{p_\star}(\TT_{n, \alpha}(\sigma)) \leq Cn\, \texte^{c(\alpha, p)n}\,
            \BbbP_{p_\star}(\AA_{n, 0}).
\end{equation}
But on $\AA_{n, 0}$, the connected component of the origin has at least $n$ edges, and thus order~$n$ vertices. By \eqref{E:2.10b},  $\BbbP_{p_\star}(\AA_{n, 0})\le\texte^{-c'(p) n}$ for some $c'(p)>0$. Noting that $c(p,\alpha)\downarrow0$ as $\alpha\downarrow0$, choosing~$\alpha$ sufficiently small, we obtain exponential decay for $\BbbP_{p_\star}(\AA_{n, \alpha})$.
\end{proof}

\subsection{Proof of Theorem~\ref{thm:BetaP} and Proposition~\ref{prop:SymmetryForBeta}}
The proof of the theorem will come in several parts. First we will establish the existence of the limit in \eqref{eqn:902} along multiples of integers.

\begin{lemma}
\label{lem:BetaPExistenceOnZ2}
Let $p>p_\cc(\Z^2)$. Then for all $x \in \Z^2$, the limit in \eqref{eqn:902} exists pointwise $\BbbP$-a.s.\ and in $L^1$ and is non-random and finite.
\end{lemma}

\begin{proof}
For $x=0$ the claim is trivial so fix some $x \in \Z^2 \setminus\{0\}$ and for $0 \leq m < n$ define
\begin{equation}
	b_{m,n} := b\bigl([mx], [nx]\bigr).
\end{equation}
By virtue of Lemma~\ref{prop:MergingOfPaths}, for any $x,y,z$ that are connected in $\omega$,
\begin{equation}
\label{eqn:905}
	b(x,z) \leq b(x,y) + b(y,z) + 2.
\end{equation}
It thus follows 
\begin{equation}
\label{E:2.19}
	b_{0,n} \leq b_{0,m} + b_{m,n} + 2,\qquad 0\le m<n.
\end{equation}
This puts us in a position to extract the limit by subadditivity arguments. 

We will specifically rely on Liggett's version~\cite{Liggett} of Kingman's Subadditive Ergodic Theorem 
which states that if $(X_{m,n})_{0\le m< n}$ are non-negative random variables satisfying
\settowidth{\leftmargini}{(11111)}
\begin{enumerate}
\item[(1)] $X_{0,n}\le X_{0,m}+X_{m,n}$ for all $0<m<n$,
\item[(2)] $\{X_{nk,(n+1)k}\colon n\ge1\}$ is stationary and ergodic for each $k\ge1$,
\item[(3)] the law of $\{X_{m,m+k}\colon k\ge1\}$ is independent of $m\ge1$, and
\item[(4)] $E X_{0,1}<\infty$,
\end{enumerate}
then the limit $\lim_{n\to\infty}X_{0,n}/n$ exists a.s.\ and in~$L^1$ and equals $\lim_{n\to\infty}E X_{0,n}/n < \infty$ almost surely. We proceed to check the conditions for the case at hand.

Define $X_{m,n}:=b_{m,n}+2$. Then~(1) follows from \eqref{E:2.19}. For~(2) let us write $b(x,y;\omega)$ to explicate the dependence on~$\omega$ and let $\tau_z\omega$ denote the shift of~$\omega$ by~$z$. 
By the definition of $[v]$, we have
\begin{equation}
\label{eqn:904}
	b\bigl([u + z]_\omega, [v + z]_\omega; \omega\bigr) = 
		b\bigl([u]_{\tau_z\omega}, [v]_{\tau_z\omega}, \tau_z\omega\bigr) \ ;\quad u,v,z \in \Z^2 ,
\end{equation}
where the subindex $[u]_\omega$ indicates which percolation configuration the bracket is taken in.
From here we get $X_{m,n}=X_{0,n-m}\circ\tau_x^m$; the conditions~(2,3) then hold by the translation invariance and ergodicity of the law~$\BbbP$ under the shift $\tau_x$. As to~(4), here we note that, by Lemma~\ref{prop:DualPathLength},
\begin{equation}
\label{eqn:906}
	b\bigl([x], [y];\omega\bigr) \leq 3 D_\omega\bigl([x],
		[y]\bigr) \,.
\end{equation}
From here we get~(4) using Lemma~\ref{prop:ImprovedPisztora}. The claim immediately follows.
\end{proof}

\begin{remark}
We note that the use of Subadditive Ergodic Theorem is not required for proving the convergence \eqref{eqn:902} in the mean. This would come at no loss as almost sure convergence could then be extracted from the concentration bounds in Theorem~\ref{thm:ConcentrationOfB}. However, we find the fact that almost-sure convergence can be proved using soft methods more appealing.
\end{remark}

Next we will show that $\beta_p$ is {\em positive} homogeneous and sub-additive on~$\Z^2$:

\begin{lemma}
For $\beta_p$ as in Lemma~\ref{lem:BetaPExistenceOnZ2},
\begin{equation}
\label{E:hom}
\beta_p(nx)=n\beta_p(x),\qquad x\in\Z^2,\,n\in\N,
\end{equation}
and
\begin{equation}
\label{E:triangle}
\beta_p(x+y) \leq \beta_p(x) + \beta_p(y),\qquad x,y\in\Z^2.
\end{equation}
\end{lemma}

\begin{proofsect}{Proof}
The positive homogeneity \eqref{E:hom} follows from the very existence of the limit. For the triangle inequality \eqref{E:triangle} fix $x,y \in\Z^2$ and note that for all $n \geq 1$, by \eqref{eqn:905} and~\eqref{eqn:904},
\begin{equation}
\begin{split}
	\E\, b\bigl([0], [n(x+y)]\bigr) \ & \leq 
		\E\, b\bigl([0], [nx]\bigr) + \E\, b\bigl([nx], [n(x+y)]\bigr) + 2 \\	
	& =
		\E\, b\bigl([0], [nx]\bigr) + \E\, b\bigl([0],[ny]\bigr) + 2 .
\end{split}
\end{equation}
Dividing both sides by $n$ and taking $n \to \infty$ we obtain \eqref{E:triangle} as desired.
\end{proofsect}

\begin{lemma}
\label{lemma-2.13}
Let $\beta_p$ be as in Lemma~\ref{lem:BetaPExistenceOnZ2}. Then for all $(x_1, x_2) \in \Z^2$,
\begin{equation}
	\beta_p\bigl((x_1, x_2)\bigr)=\beta_p\bigl((x_2, x_1)\bigr)=\beta_p\bigl((\pm x_1,\pm x_2)\bigr)
\end{equation}
for any choice of the two signs~$\pm$. 
\end{lemma}

\begin{proof}
We will prove that these symmetries hold in distribution for $b([x],[y])$. Since $\BbbP$ (and the definition of~$[x]$) is invariant under rotations by $90^\circ$, since right boundaries (and right-most paths) remain such under these rotations, for all $x_1, x_2 \in \Z$ we have
\begin{equation}
	b\bigl([0], [(x_1, x_2])\bigr) 
		\, \eqd \, b\bigl([0], [(-x_2, x_1)]\bigr)  
		\, \eqd \, b\bigl([0], [(-x_1, -x_2)]\bigr)  
		\, \eqd \, b\bigl([0], [(x_2, -x_1)]\bigr).  
\end{equation}
Using this together with the invariance of~$\BbbP$ with respect to translations, one obtains
\begin{equation}
\label{E:5.16}
 b\bigl([(x_1,x_2)],[0]\bigr)\,\eqd\, b\bigl([0],[(-x_1,-x_2)]\bigr)\,\eqd\, b\bigl([0],[(x_1,x_2)]\bigr) .
\end{equation}
As to the reflections through the axes, while $\BbbP$ (and the definition of~$[x]$) is still invariant, right-boundaries (and right-most paths) reflect into left-boundaries and left-most paths (which are the obvious ``left'' counterparts of our standard ``right'' objects). However, a left-most path becomes right-most when travelled backwards and so we get
\begin{equation}
	b\bigl([0],[(x_1,x_2)]\bigr) 
	\,\eqd\,b\bigl([(x_1, -x_2)], [0] \bigr) 
	\,\eqd\,b\bigl([(-x_1, x_2)], [0] \bigr) .
\end{equation}
Combining all of the above we see that the law of $x\mapsto  b([0],[x])$ has all of the stated symmetries. Then $\beta_p$ inherits all these symmetries as the $L^1$ limit in \eqref{eqn:902}.
\end{proof}

In particular, the above two lemmas show that $\beta_p$ is homogeneous on $\Z^2$. 
As such it admits a well-defined and unique homogeneous extension to a function~$\beta_p\colon\Q^2\to[0,\infty)$ via
\begin{equation}
\label{E:2.25}
	\beta_p(x/q) := \beta_p(x)/q,  \qquad x \in \Z^2,\, q \in \N.
\end{equation}
We then note:

\begin{lemma}
\label{lemma-2.10}
The function $\beta_p\colon\Q^2\to[0,\infty)$ from \eqref{E:2.25} extends continuously to~$\R^2$. This extension is homogeneous and satisfies the triangle inequality as stated in (1-2) of Theorem~\ref{thm:BetaP}. Moreover,
\begin{equation}
\label{E:2.26}
\sup\biggl\{\,\frac{\beta_p(x)}{\Vert x\Vert} \colon  x \in \R^2 \setminus \{0\}\,\biggr\}<\infty.
\end{equation}
\end{lemma}

\begin{proofsect}{Proof}
It is easy to check that $\beta_p$ is homogeneous and obeys the triangle inequality on~$\Q^2$. This (and the finiteness of $\beta_p$ along, say, the coordinate directions) imply that $x\mapsto\beta_p(x)/\Vert x\Vert$ is bounded on~$\Q^2\setminus\{0\}$. In particular, $\beta_p(x)\to0$ as $x\to0$ on~$\Q^2$. Thanks to the triangle inequality, $x\mapsto\beta_p(x)$ is Lipschitz-continuous on~$\Q^2$ and, by the denseness of~$\Q^2$ in~$\R^2$, it can be extended continuously to all of~$\R^2$. The bound \eqref{E:2.26} is inherited from the same bound on~$\Q^2\setminus\{0\}$.
\end{proofsect}

Next we will show that~$\beta_p$ is non-degenerate on the unit circle:

\begin{lemma}
\label{lem:BetaPIsNonTrivial}
For~$\beta_p$ from Lemma~\ref{lemma-2.10},
\begin{equation}
\label{E:2.30d}
\inf \biggl\{\,\frac{\beta_p(x)}{\Vert x\Vert} \colon  x \in \R^2 \setminus \{0\}\,\biggr\}> 0.
\end{equation}

\end{lemma}
\begin{proof}
Invoking continuity and homogeneity, it suffices to prove this on~$\Z^2\setminus\{0\}$. Fix an $x$ in this set. The union bound then yields
\begin{multline}
\qquad
\BbbP \bigl(b([0], [nx]) \le\alpha n /3\big) \leq \BbbP\bigl(\|[0]\| > \ffrac n3\bigr) + \BbbP\bigl(\|[nx] - nx\| > \ffrac n3\bigr) \\ 
\qquad+ \sum_{\|y\| \leq n/3} \BbbP\bigl(\exists \gamma \in \cup_{x \in \Z^2} \RR(y,x) \colon  \big|\gamma| \geq \ffrac n3,\, \frb(\gamma) \leq \alpha n / 3\bigr).
\quad
\end{multline}
Invoking Lemma~\ref{prop:DistanceToCInf} and Proposition~\ref{prop:RenormalizedLD} and setting $\alpha:=\alpha_2$, the right-hand side decays exponentially in~$n$ with all constant uniform in $x \in \Z^2 \setminus\{0\}$. It follows that $\E b([0], [nx]) \geq C n$ for some $C > 0$ independent of $x$. Dividing by $n$ and taking $n \to \infty$, we get \eqref{E:2.30d}.
\end{proof}

\begin{proof}[Proof of Theorem~\ref{thm:BetaP}]
The existence of~$\beta_p$ as a norm, as well as the limit \eqref{eqn:902} for~$x\in\Z^2$, was established in the above lemmas; it remains to verify that the limit \eqref{eqn:902} applies to this extension for $x$ that are not on the lattice. To this end, let $\epsilon\in(0,2\pi)$, define $N:= \lfloor 2\pi/\epsilon \rfloor$ and let $\hat{u}_0 \dots, \hat{u}_N \in \Q^2$ be such that
\begin{equation}
	\| \hat{u}_k - \texte^{\texti k\epsilon} \|_2\leq \epsilon/2,\qquad k=0,\dots,N,
\end{equation}
where $\texte^{\texti k\epsilon}:=(\cos(k\epsilon),\sin(k\epsilon))$.
Finally, choose $M\in\N$ so that $M \hat{u}_k \in \Z^2$ for all $k$.

From the construction, for each $x \in \R^2 \setminus \{0\}$, we may find $k$ such that 
\begin{equation}
\label{eqn:907}
\hat{x} := x / \|x\|_2	 \quad\text{obeys}\quad\|\hat{u}_k-\hat{x}\|_2 \leq \epsilon .
\end{equation}
Setting $y_n  := \lfloor n\|x\|_2/M \big\rfloor M\hat{u}_k$, we have $y_n \in \Z^2$ and 
\begin{equation}
	\big\| y_n - nx \big\|_2 
		\ \leq \  n\|x\|_2 \|\hat{u}_k - \hat{x}\|_2 + C M
		\ \leq \ n\epsilon \|x\|_2 + C M  .
\end{equation}
Then using Lemma~\ref{prop:ImprovedPisztora} for some $\alpha > 0$ and all $r > \alpha (\epsilon \|x\|_2 + M/n)$, we observe
\begin{equation}
\label{E:2.35}
\BbbP\bigl(D_\omega([nx], [y_n])/n > r\bigr) \leq C\texte^{-C'nr}
\end{equation}
and so by Borel-Cantelli,
\begin{equation}
\label{eqn:906.1}
\limsup_{n \to \infty} D_\omega([nx], [y_n])/n 
	\leq \alpha \epsilon \|x\|_2,\qquad \BbbP\text{-a.s.}.
\end{equation}
Similarly, we also get
\begin{equation}
\label{eqn:906.2}
\E \bigl(D_\omega([nx], [y_n])/n\bigr)
	\leq \alpha (\epsilon \|x\|_2 + M/n) + C/n \, .
\end{equation}
On the other hand, by \eqref{eqn:905} and \eqref{eqn:906.1},
\begin{equation}
\begin{aligned}
\label{eqn:909}
	\bigl|\, b([0], [nx]) - b([0], [y_n]) \, \bigr| 
		&\leq \ b([nx], [y_n]) + 2
		\\&\leq \ 3D_\omega([nx], [y_n]) + 2 .
\end{aligned}
\end{equation}
Lemma~\ref{lem:BetaPExistenceOnZ2} then yields
\begin{equation}
\label{eqn:908}
	\lim_{n \to \infty}  \frac{b([0], [y_n])}{n}
		= \|x\|_2 \beta_p(\hat{u}_k),\qquad \BbbP\text{-a.s and in }L^1.
\end{equation}
Finally, \eqref{eqn:907} and the fact that $\beta_p$ is an equivalent norm imply also
\begin{equation}
\label{eqn:910}
	\big|\|x\|_2 \beta_p(\hat{u}_k) - \beta_p(x) \big| \leq C \epsilon \|x\|_2
\end{equation}
for some $C>0$.

Combining \twoeqref{eqn:906.1}
{eqn:910} we obtain both $\BbbP$-almost surely and in $L^1$,
\begin{equation}
\label{eqn:911}
\limsup_{n \to \infty} \left | \frac{b([0], [nx])}{n} - \beta_p(x) \right | \leq (3\alpha + C)\epsilon \|x\|_2 ,
\end{equation}
and since this is true for all $\epsilon > 0$, the existence of the limit for general $x \in \R^2$ is established.

To justify the claim about uniformity, note that there are only finitely many $\hat{u}_k$'s for a given $\epsilon>0$ and there are only at most order-$n$ distinct values of~$[nx]$ when~$x$ ranges through the unit circle in~$\R^2$. The former ensures uniformity of \eqref{eqn:908}, the latter implies that \eqref{E:2.35} still holds when the quantifier $\forall x\in\{z\colon\Vert z\Vert_2=1\}$ is inserted into the probability on the left. The conclusions \twoeqref{eqn:906.1}{eqn:906.2} then also hold (for a given~$\omega$) uniformly in~$x$ with~$\Vert x\Vert_2=1$ and, consequently, so does the limsup in \eqref{eqn:911}. The same applies to $L^1$ convergence.
\end{proof}

\begin{proof}[Proof of Proposition~\ref{prop:SymmetryForBeta}]
The symmetries of $\beta_p$ as function on $\Z^2$ as proved in Lemma~\ref{lemma-2.13} are preserved under its extension to all of~$\R^2$.
\end{proof}

\section{Concentration estimates}
\label{sec3}\noindent
Theorem~\ref{thm:BetaP} defines the boundary norm $\beta_p$ as the almost sure limit \eqref{eqn:902}. However, our applications will require control of the rate of convergence which we achieve by proving the following  concentration estimate:

\begin{theorem}[Concentration for right-boundary distance]
\label{thm:ConcentrationOfB}
Let $p>p_\cc(\Z^2)$. For each $\epsilon > 0$ there is $C > 0$ and $N>0$ such that for all $x,y \in \Z^2$ with $\Vert x-y\Vert>N$,
\begin{equation}	
\label{E:3.1}
	\BbbP \biggl(\left|\frac{b([x],[y])}{\beta_p(y-x)} - 1 \right|	> \epsilon \biggr) \leq \texte^{-C\log^2 \|y-x\|} 
	.
\end{equation}
\end{theorem}

Apart from this (measure theoretic) concentration estimate for the value of~$b(x,y)$, in order to prove existence of the limiting shape of the isoperimetric sets, we will need to control the geometric concentration of the paths minimizing, or nearly minimizing, $\frb(\gamma)$. To this end, whenever $x,y\in\mbC^\infty$, let us call the path $\gamma \in \RR(x,y)$ \emph{$\epsilon$-optimal} if
\begin{equation}
	\frb(\gamma) - b(x,y) \leq \epsilon \|y-x\| .
\end{equation}
We will write $\Gamma_\epsilon(x,y)$ for the set $\epsilon$-optimal paths in~$\RR(x,y)$; the (absolute) minimizers then constitute the set $\Gamma_0(x,y)$. Note that, since $\frb(\gamma)$ is integer valued, $\Gamma_0(x,y)\ne\emptyset$ $\BbbP$-a.s.\ for any $x,y\in\mbC^\infty$ (assuming, of course, $p>p_\cc(\Z^2)$). Recall also the notation $d_{\rm H}(A,B)$ from \eqref{E:dH}.

\begin{proposition}
\label{prop:GeomConcentration}
Let $p>p_\cc(\Z^2)$. There are $\alpha,C,C' >0$ such that for all $x,y \in \Z^2$, we have:
\begin{enumerate}
\item \label{part:GC-1}
For any $t > \alpha \|x-y\|$,
	\begin{equation}
	\label{eqn:921}
	\BbbP\bigl(\exists \gamma \in \Gamma_0([x], [y]) \colon |\gamma| > t\bigr) \leq C\texte^{-C' t}.
	\end{equation}
\item \label{part:GC-2}
	For all $\epsilon > 0$, once $\|y-x\|$ is sufficiently large (depending on $\epsilon$),
		\begin{equation}
		\label{eqn:922}	
		\BbbP\bigl(\forall \gamma \in \Gamma_\epsilon([x], [y]) \colon d_{\rm H}(\gamma,\, \poly(x,y)) > \epsilon \|y-x\|\bigr) \leq
				C\texte^{-C' \log^2\|y-x\|} ,
		\end{equation}
		where $\poly(x,y)$ is the linear segment connecting $x$ and $y$.
\end{enumerate}
\end{proposition}

\begin{remark}
The estimates in \eqref{E:3.1} and \eqref{eqn:922} are stated in the form which is sufficient for the purposes of this paper, but they are far from optimal as far as the actual decay goes. Indeed, with modest changes to the proofs one should be able to improve these into $\texte^{-C\|y-x\|^{1/2-\epsilon}}$. (The bottleneck is Proposition~\ref{lem:BApproxBHat}, where one would need to replace the penalty function $h(t)$ by something that is nearly linear in~$t$.) However, we believe that even this may not be optimal.
\end{remark}

The remainder of Section~\ref{sec3} is devoted to the proofs of Theorem~\ref{thm:ConcentrationOfB} and Proposition~\ref{prop:GeomConcentration}. The underlying idea is to write $b([x],[y])-\E\,b([x],[y])$ as a Doob martingale and apply an Azuma-type concentration estimate. Unfortunately, such estimates generally require a representation using a martingale with \emph{bounded} increments which, due to the requirement that the optimizing paths be open in the underlying percolation configuration, is not the case for the random variable~$b([x],[y])$. We will thus have to work with a modified right-boundary distance $\hat b(x,y)$ that has this property. This is a quantity similar to $b(x,y)$; the principal difference is that it allows, at a huge penalty, inclusion of non-fully open paths. This permits us to invoke a concentration estimate from Kesten's study~\cite{kesten1993speed} of the shortest-time paths in first-passage percolation (see Theorem~\ref{thm-Kesten} in Section~\ref{sec3.2}).

\subsection{Modified right-boundary distance}
We proceed to introduce the modified right-boundary distance $\hat b(x,y)$ and derive a number of its properties. Throughout we assume $p>p_\cc(\Z^2)$. For a path $\gamma$ we set 
\begin{equation}
	\frp(\gamma):= 
	\bigl|\{ e \in \gamma \colon \omega(e) =  0\}\bigr|,
\end{equation} 
where the opposite orientations of an edge are considered distinct.
The modified version $\hat{\frb}(\gamma)$ of $\frb(\gamma)$ is defined for a right-most, but not necessarily open, path $\gamma$ with endpoints~$x$ and~$y$ by
\begin{equation}
	\hat{\frb}(\gamma) 
	: =    \frb(\gamma) + h(\|y - x\|) \frp(\gamma) ,
\end{equation}
where we set, once and for all,
\begin{equation}
\label{E:3.7}
	h(t) := \max\bigl\{\log^4 t,1\}.
\end{equation}
For $x,y \in \Z^2$ we then define the \textit{modified boundary distance}:
\begin{equation}
\label{eqn:920}
    \hat{b}(x, y) 
    := \inf \bigl\{\hat{\frb}(\gamma) \colon
        \gamma \in \RR(x,y)\bigr\}.
\end{equation}
The set of minimizers in \eqref{eqn:920} will be denoted by $\hat{\Gamma}(x,y) = \hat{\Gamma}(x,y ;\; \omega)$.

The first lemma controls the length of any optimal path for $b(x,y)$ and $\hat{b}(x,y)$. Part~\ref{part:BOOP-3} is identical to part~\ref{part:GC-1} of Proposition~\ref{prop:GeomConcentration}; we restate it here so that we can include its proof already at this point.

\begin{lemma}
\label{lem:BoundsOnOptPath}
For $p>p_\cc(\Z^2)$, there are $\alpha > 0$, $C, C' > 0$ such that for all $x,y \in \Z^2$,
\settowidth{\leftmargini}{(1111)}
\begin{enumerate}
\item \label{part:BOOP-1}
	$\hat{b}(x,y) \leq C \|y-x\|h\bigl(\|y-x\|\bigr)$.
\medskip
\item \label{part:BOOP-2}
	If $t > \alpha h\bigl(\|y-x\|\bigr)\|y-x\|$, then
	\begin{equation}
		\BbbP(\exists \gamma \in \hat{\Gamma}(x,y)
			\colon |\gamma| > t) \leq C \texte^{-C' t} .
	\end{equation}
\item \label{part:BOOP-3}
	If $t > \alpha \|y-x\|$, then
	\begin{equation}
	\label{eqn:933}
		\BbbP(\exists \gamma \in \Gamma_0([x],[y])
			\colon |\gamma| > t) \leq C \texte^{-C' t} .
	\end{equation}
\end{enumerate}
\end{lemma}

\begin{proof}
For (1), let $\gamma$ be any path (not necessarily open) from $x$ to $y$ that is shortest in the lattice distance. Since $\frb(\gamma),\frp(\gamma) \, \le\,  3|\gamma| \,=\, 3\|x-y\|$ and $h\ge1$, we get $\hat{\frb}(\gamma) \leq 6\|y-x\|\,h\bigl(\|y-x\|\bigr)$.

For part~(2), let $\alpha_2$ be as in Proposition~\ref{prop:RenormalizedLD}. If $t > 6 (\alpha_2)^{-1} \|y-x\|\,h\bigl(\|y-x\|\bigr)$, then by part~(1) any $\gamma\in\hat\Gamma(x,y)$ obeys $\frb(\gamma)\le\hat{\frb}(\gamma)<\alpha_2t$. Hence,\begin{equation}
\begin{aligned}
\BbbP\bigl(\exists \gamma \in \hat{\Gamma}(x,y)\colon |\gamma| > t\bigr) &\leq
\BbbP\bigl(\exists \gamma \in \RR(x,y) \colon |\gamma| > t,\, \frb(\gamma) <\alpha_2t\bigr),
\end{aligned}
\end{equation}
which decays exponentially in~$t$ by Proposition~\ref{prop:RenormalizedLD}. Part~(2) thus holds with $\alpha:=(\alpha_2)^{-1}$.

Lastly, for part~(3), by Lemma~\ref{prop:ImprovedPisztora} and
\eqref{eqn:906}, if $t > 3 \alpha_1 \|y-x\|$
\begin{equation}
\BbbP\bigl(b([x], [y]) > t\bigr)
	\leq \BbbP\bigl(3D_\omega([x], [y]) > t\bigr) \leq C' \texte^{-C'' t} .
\end{equation}
But, on $\{b([x], [y]) \leq t\}\cap\{\exists \gamma \in \Gamma([x],[y]) \colon |\gamma| > (\alpha_2)^{-1} t\}$  there \emph{is} a path
$\gamma \in \RR([x],[y])$ with $|\gamma| > (\alpha_2)^{-1} t$ and  $\frb(\gamma) \leq t$. By Proposition~\ref{prop:RenormalizedLD} this has probability at most $C' \texte^{-C''t}$. Therefore, we get \eqref{eqn:933} as soon as $t > 3 \alpha_1 (\alpha_2)^{-1} \|y-x\| + C$.
\end{proof}

Next we wish to argue that with high probability, the quantities $\hat{b}(x,y)$ and $b([x], [y])$ are (relatively) close to each other once $\Vert x-y\Vert$ is large:

\begin{proposition}
\label{lem:BApproxBHat}
For $p>p_\cc(\Z^2)$, there are $C, C' > 0$, such that for all $x,y \in \Z^2$ with $N:=\|x-y\|$ large enough,
\begin{equation}
\label{eqn:YKW1}
	\BbbP \Bigl(\,\bigl|b([x], [y]) - \hat{b}(x,y)\bigr|
	\geq C\, h(N)\log^2(N) \Bigr) \leq \texte^{-C' \log^2(N)}.
\end{equation}
\end{proposition}

The core of the proof boils down to the following observation:

\begin{lemma}
\label{lemma-3.6}
For $p>p_\cc(\Z^2)$, there is $C>0$ such that for $\Vert x-y\Vert$ large enough,
\begin{equation}
\label{E:YKW}
\BbbP\Bigl(\exists\gamma \in \hat{\Gamma}([x], [y]),\, \text{\rm not open}\Bigr)\le \texte^{-C\log^2\Vert x-y\Vert}.
\end{equation}
\end{lemma}

\begin{proofsect}{Proof}
Set $N:=\Vert x-y\Vert$ and let $\AA_N$ be the event that $[x]$ and~$[y]$, as well as any $\gamma \in \hat{\Gamma}([x],[y])$, stay inside the box $x+[-N^2,N^2]^2$ and that any dual connected component intersecting this box is circumnavigated by a self-avoiding circuit of open edges whose length does not exceed $\log^2(N)$. Thanks to Lemmas~\ref{prop:DistanceToCInf}, \ref{lem:BoundsOnOptPath} and the exponential decay of dual connectivities, $\BbbP(\AA_N)\ge 1-\texte^{-C\log^2(N)}$ once $N$ is sufficiently large.

Assume now that $\AA_N$ occurs and that some $\gamma \in \hat{\Gamma}([x],[y])$ contains a closed edge~$e$. Then its dual edge $e_\star$ is part of a dual connected component and is thus surrounded by an open self-avoiding (and, in particular, right-most) circuit $\lambda$ of length at most~$\log^2(N)$. Suppose first that $[x]$ and~$[y]$ do not lie in the interior of $\lambda$. Then $\gamma$ would visit at least one vertex of $\lambda$; we set $z$, respectively, $w$ to  the first, respectively, last vertex in $\lambda$ visited by $\gamma$. Let $\gamma_L$ denote the sub-path of $\gamma$ that connects $[x]$ to $z$, write $\gamma_R$ for the sub-path of $\gamma$ that connects $w$ to $[y]$ and let $\lambda_M$ denote the sub-path of $\lambda$ which connects $z$ to $w$ (which may coincide). Define $\gamma' := (\gamma_L * \lambda_M )* \gamma_R\in\RR([x],[y])$ and note
\begin{equation}
\label{eqn:941}
\hat{\frb}(\gamma')\le\hat{\frb}(\gamma)-h\bigl(\|y-x\|\bigr)+3|\lambda|.
\end{equation}
Since $3|\lambda|\le3\log^2(\|y-x\|)<h(\|y-x\|)$ this would mean that $\gamma$ is not optimal. 

If, on the other hand, $[x]$ does lie in the interior of $\lambda$, then we set $z$ to be the first vertex in $\lambda$ that is on some self avoiding open path from $[x]$ to $\infty$ (chosen according to some a priori ordering of paths). Such a path must exist since $[x] \in \mbC^\infty$; we denote its segment from $[x]$ to $z$ by~$\gamma_L$. With the remaining paths defined as before, we see that \eqref{eqn:941} still holds, which leads to a contradiction again. The case of $[y]$ in the interior of $\lambda$ (alone or together with $[x]$) is treated in the same way. The probability in \eqref{E:YKW} is thus bounded by $\BbbP(\AA_N^\cc)$.
\end{proofsect}

\begin{proof}[Proof of Proposition~\ref{lem:BApproxBHat}]
Let $x,y\in \Z^2$ be such that $\|x-y\|$ is large enough for the arguments to follow. By Lemma~\ref{lemma-3.6}, with probability at least $1-\texte^{-C \log^2 \|y-x\|}$, any path $\gamma \in \hat{\Gamma}([x], [y])$ must be open. When this is the case, we clearly have 
\begin{equation}
\label{eqn:944}
	b\bigl([x], [y]\bigr) = \hat{b}\bigl([x], [y]\bigr)
\end{equation}
and so we only need to worry about the difference $\hat b(x,y)-\hat{b}([x], [y])$. By Lemma~\ref{prop:DistanceToCInf},
with probability at least $1-\texte^{-C\log^2 \|y-x\|}$ 
\begin{equation}
\label{eqn:942}
	\bigl\|[x]-x\bigr\| < \log^2 \|y-x\| \quad\text{and} \quad\bigl\|[y]-y\bigr\| < \log^2 \|y-x\|.
\end{equation}
For any $u,v,w \in \Z^2$, by $\ast$-concatenating a shortest (right-most) lattice path from~$u$ to~$v$  with
a path $\gamma'\in\hat\Gamma(v,w)$, Lemma~\ref{lem:BoundsOnOptPath}(1) yields
\begin{equation}
\hat{b}(u,w) \leq C \|v-u\| h(\|w-u\|) + \hat{b}(v,w)\biggl( \frac{h(\|w-u\|)}{h(\|w-v\|)} \vee\  1\biggr) .
\end{equation}
To apply this bound, we will set $u:=x$, $v:=[x]$ and $w:=y$.
Thanks to the definition of~$h$ in \eqref{E:3.7} and the bounds in \eqref{eqn:942} we then have
\begin{equation}
	\frac{h(\|y-x\|)}{h(\|y-[x]\|)} - 1 \leq C'\, \frac{\log \|y-x\|}{\|y-x\|}.
\end{equation}
Using also Lemma~\ref{lem:BoundsOnOptPath}(1), we thus get
\begin{equation}
\hat{b}\bigl([x],y\bigr)\biggl( \frac{h(\|y-x\|)}{h(\|y-[x]\|)} \vee\  1\biggr)
\le\hat{b}\bigl([x],y\bigr)+C'C\frac{\Vert y-[x]\Vert h\bigl(\Vert y-[x]\Vert\bigr)}{\Vert y-x\Vert}\log\Vert y-x\Vert.
\end{equation}
Redefining the meaning of~$C$, we then conclude
\begin{equation}
	\hat{b}(x,y) \leq \hat{b}\bigl([x],y\bigr) + C\, h\bigl(\|y-x\|\bigr)\log^2 \|y-x\| .
\end{equation}
By further comparing $\hat{b}([x], y)$ and $\hat{b}([x], [y])$ and reversing the roles of
$\hat{b}(x,y)$ and $\hat{b}([x], [y])$, we thus get 
\begin{equation}
\label{eqn:3.21}
	\bigl|\hat{b}(x,y) - \hat{b}([x], [y])\bigr| \leq 
		C \,h\bigl(\|y-x\|\bigr)\log^2 \|y-x\| \,.
\end{equation}
Then, \eqref{eqn:YKW1} follows from 
\eqref{eqn:944} and \eqref{eqn:3.21}.
\end{proof}

It will now come as no surprise that $\hat{b}(x,y)$ and $b([x], [y])$ are also close in expectation.

\begin{lemma}
\label{lem:BApproxBHatInMean}
For $p>p_\cc(\Z^2)$ there is $C > 0$, such that for all $x,y \in \Z^2$ with
$\|y-x\|$ large,
\begin{equation}
\label{E:3.22}
	\bigl| \,\E\, b([x], [y]) - \E \,\hat{b}(x,y) \bigr| 
	\leq C h\bigl(\|y-x\|\bigr)\log^2\|y-x\|.  
\end{equation}
\end{lemma}

\begin{proof}
Fix $x,y$ such that $\|y-x\|$ is large enough and set
$r := C h(\|y-x\|)\log^2\|y-x\|$, where $C$ is as in 
Proposition~\ref{lem:BApproxBHat}. By Cauchy-Schwarz,
\begin{equation}
\begin{aligned}
\E \bigl|b([x], & [y]) - \hat{b}(x,y) \big| \\
	& \leq r + \E \Bigl(\bigl|b([x], [y]) - \hat{b}(x,y)\bigr| 
		\ind_{\{|b([x], [y]) - \hat{b}(x,y)| > r\}} \Bigr)\\
	& \leq 	 r + C'' \texte^{-C'\log^2\|y-x\|}
			\bigl (\E b([x], [y])^2 + 
				\E \hat{b}(x,y)^2 \bigr) .
\end{aligned}
\end{equation}
Now it follows from \eqref{eqn:906}, Lemma~\ref{prop:ImprovedPisztora} and
Lemma~\ref{lem:BoundsOnOptPath}(1) that 
\begin{equation}
\E\bigl(b([x], [y])^2\bigr) \leq \tilde C \|y-x\|^2
	\quad \text{and} \quad 
\E\bigl( \hat{b}(x,y)^2 \bigr)\leq \tilde C \|y-x\|^2 h((\|y-x\|))^2
\end{equation} 
for some $\tilde C>0$. The left hand-side in \eqref{E:3.22} is then bounded by~$2r$ as soon as $\|y-x\|\gg1$.
\end{proof}

\subsection{Concentration for modified right-boundary distance}
\label{sec3.2}\noindent
Having controlled the differences between right-boundary distance and its modified counterpart, we now proceed to show that~$\hat{b}(x,y)$ concentrates stretched exponentially around its mean. 

\begin{proposition}
\label{lem:BConcentration}
Suppose $p>p_\cc(\Z^2)$. There are constants $C, C'>0$ such that for any $x,y \in \Z^2$ with
$\|x-y\|$ large enough and any $t\leq \|x-y\|^\frac{3}{2}$,
\begin{equation}
    \BbbP \Bigl( \bigl|\hat{b}(x, y)-\E\,\hat{b}(x, y)\bigr| > t\Bigr) \leq   
    	C \exp \left(-\frac{C't}{\|x-y\|^{1/2} (h(\|x-y\|))^{3/2}}\right).
\end{equation}
\end{proposition}

Note that $\hat{b}(x, y)$ is bounded by Lemma~\ref{lem:BoundsOnOptPath}(1) and so the expectation exists.
As already alluded to, the proof will be based on a concentration estimate for martingales with bounded increments. (In fact, the boundedness of increments is the principal reason why we state this for $\hat b(x,y)$ rather than~$b(x,y)$.) Such estimates have appeared in various forms in the literature; here we will invoke Theorem 3 in Kesten~\cite{kesten1993speed} whose slightly simplified form reads as:

\begin{theorem}[Kesten~\cite{kesten1993speed}]
\label{thm-Kesten}
Let $\{\scrF_k\}_{k=0}^{\infty}$ be an increasing family of $\sigma$-algebras and let us set $\scrF_\infty:=\bigvee
_{k=0}^\infty \scrF_k$. Suppose that $(M_k;\scrF_k)_{k=0}^\infty$ is a martingale whose increments $\Delta_k : = M_k - M_{k-1}$, $k=1,2,\dots$, obey
\begin{equation}
\label{eqn:951}
    |\Delta_k|^2 \leq \alpha \quad \text{and} \quad
    E[\Delta_k^2|\scrF_{k-1}]\leq E[U_k|\scrF_{k-1}],\qquad k\ge1,
\end{equation}
for some $\alpha > 1$ and for some sequence $\{U_k\}_{k=0}^\infty$ of positive $\scrF_\infty$-measurable random variables that satisfy
\begin{equation}
\label{eq:cdown}
    P \Bigl(\,\sum_{k=1}^\infty U_k > t \Bigr)
    	\leq C \texte^{-C' t},\qquad t \geq \alpha,
\end{equation}
for some $C, C' > 0$. Then, $M_\infty:=\lim_{n\rightarrow\infty}M_n$ exists almost surely and there is $C_1 > 0$, which may depend on $C, C'$, and a universal $C_2 > 0$, such that
\begin{equation}
\label{eqn:953}
   P\bigl(|M_\infty - M_0| \geq t\bigr)
    	\leq C_1 \exp (-C_2 t/ \sqrt{\alpha} ),\qquad t \leq C' \alpha^{3/2}.
\end{equation}
\end{theorem}

\begin{proofsect}{Proof}
This is a shortened version of the statement of Theorem 3 in Kesten~\cite{kesten1993speed} bypassing the sharper, but less illuminating, estimate (1.28). Our bound \eqref{eqn:953}  corresponds to Kesten's equation (1.29) simplified with the help of~$\alpha>1$.
\end{proofsect}

\begin{proofsect}{Proof of Proposition~\ref{lem:BConcentration}}
Let $\{e_k \colon  k=0, \dots\}$ be a fixed ordering of the edges of~$\Z^2$. Define the sigma-algebras $\scrF_k:=\sigma(\omega(e_1), \dots, \omega(e_k))$. We will apply Theorem~\ref{thm-Kesten} to the martingale \begin{equation}
M_k := \E\bigl(\,\hat{b}(x,y) \,\big|\, \scrF_k\bigr)
\end{equation}
for some fixed $x,y \in \Z^2$. Our main task is to verify the conditions \twoeqref{eqn:951}{eq:cdown}.

Let us write $\hat b(x,y;\omega)$ to explicate the dependence of this object on the underlying percolation configuration~$\omega$ and let
\begin{equation}
\label{E:3.30}
g_k(\omega):=\int_{\{0,1\}}\BbbP\bigl(\textd\omega'(e_k)\bigr)
\bigl|\,\hat b(x,y;\omega')-\hat b(x,y;\omega)\bigr|,
\end{equation}
where $\omega'$ is equal to~$\omega$ except at edge~$e_k$ where it equals the integration variable~$\omega'(e_k)$.
Set $\Delta_k:=M_k-M_{k-1}$ for the martingale increment. Then
\begin{equation}
\label{EEE}
|\Delta_k|\le \E\bigl(g_k\big|\FF_{k}\bigr)
\end{equation}
thanks to the product nature of~$\BbbP$.

Recall that $\hat{\Gamma}(x,y ;\; \omega)$ is the set of minimizers of $\hat{\frb}(\gamma;\;\omega)$ among all paths in $\RR(x,y)$. In order to estimate the right-hand side of \eqref{EEE}, fix an arbitrary ordering of~$\RR(x,y)$ and for each~$\omega$, let  $\hat{\gamma}(\omega) ={\hat{\gamma}}(x,y ;\; \omega)\in\RR(x,y)$ denote the 
path in $\hat{\Gamma}(x,y ;\; \omega)$ which is the smallest in the above ordering.  For two percolation configurations  $\omega,\omega'$ that differ only in the state of the edge~$e_k$, 
\begin{equation}
\hat b(x,y;\omega')-\hat b(x,y;\omega)
	\le\hat{\frb}\bigl(\hat{\gamma}(\omega) ; \; \omega'\bigr)-\hat{\frb}\bigl(\hat{\gamma}(\omega);\; \omega\bigr)
\le\bigl(1+h(\|y-x\|)\bigr)\1_{A_k}(\omega),
\end{equation}
where
\begin{equation}
A_k:=\bigl\{\omega\colon e_k\in\hat{\gamma}(\omega)\cup\partial^+\hat{\gamma}(\omega)\bigr\}.
\end{equation}
Bounding the prefactor by~$2h(\|y-x\|)$ and interchanging the roles of~$\omega$ and~$\omega'$, we obtain
\begin{equation}
\label{E:3.34}
\bigl|\,\hat b(x,y;\omega')-\hat b(x,y;\omega)\,\bigr|
\le 2h(\|y-x\|)\bigl(\1_{A_k}(\omega)+\1_{A_k}(\omega')\bigr).
\end{equation}
This immediately gives the left condition in \eqref{eqn:951} with $\alpha:=16 h(\|y-x\|)^2$. 

For the condition on the right of \eqref{eqn:951}, we apply Jensen's inequality to \eqref{EEE} to get
\begin{equation}
\E\bigl(\,|\Delta_k|^2\big|\FF_{k-1}\bigr)\le\E\bigl(g_k^2\big|\FF_{k-1}\bigr).
\end{equation}
Using Jensen also with respect to the integration in \eqref{E:3.30}, we are naturally led to consider integrals of the right-hand side of \eqref{E:3.34} squared. Here we use
\begin{equation}
\bigl(\1_{A_k}(\omega)+\1_{A_k}(\omega')\bigr)^2\le 2(\1_{A_k}(\omega)+\1_{A_k}(\omega'))
\end{equation}
and note that the integral of $\1_{A_k}(\omega')$ over $\omega(e_k)$ and then ~$\omega'(e_k)$ yields the same result as the integeral of $\1_{A_k}(\omega)$ over $\omega(e_k)$ and then ~$\omega'(e_k)$.
It follows that
\begin{equation}
\E\bigl(\,|\Delta_k|^2\big|\FF_{k-1}\bigr)
\le16h\bigl(\|y-x\|\bigr)^2\E\bigl(\1_{A_k}\big|\FF_{k-1}).
\end{equation}
The condition on the right of \eqref{eqn:951} holds with $U_k:=16h(\|y-x\|)^2\1_{A_k}$.

Having checked \eqref{eqn:951}, we now turn to condition \eqref{eq:cdown}. Writing $h$ for $h(\|y-x\|)$ and $\hat{\gamma}$ for the path $\hat{\gamma}(x,y;\omega)$, Lemma~\ref{prop:DualPathLength} tells us
\begin{equation}
	\sum_{k=1}^\infty U_k = 16h^2 \bigl(|{\hat{\gamma}}| + |\partial^+{\hat{\gamma}}|\bigr)
		\leq C h^2 |{\hat{\gamma}}|.
\end{equation}
Therefore by Lemma~\ref{lem:BoundsOnOptPath}(2),
there is some $C_0 > 0$, such that
\begin{equation}
	\BbbP\Bigl(\,\sum_{k=1}^\infty U_k > t\Bigr) \leq C \texte^{-C' t},\qquad t > C_0 h(\|x-y\|)^3 \|x-y\|.
\end{equation}
We now reset $\alpha := C_0 h(\|x-y\|)^3 \|x-y\|$ and conclude that all the conditions in the theorem hold. The result then follows immediately from \eqref{eqn:953}.
\end{proofsect}

\subsection{Proof of Main Statements}
We are now ready to prove the main statements of this section.

\begin{proof}[Proof of Theorem~\ref{thm:ConcentrationOfB}]
Let us collect all the inequalities which have been established so far. Fix some $\epsilon > 0$ and let $x,y \in \Z^2$. From Theorem~\ref{thm:BetaP} we know that if $\|y-x\|$ is large enough (independent of the direction of $y-x$),
\begin{equation}
\label{E:3.40}
\big| \E\, b([x], [y]) - \beta_p(y-x) \big| \leq \epsilon
	\|y-x\| .
\end{equation}
Next, by Proposition~\ref{lem:BApproxBHat}, with probability at least $1-\texte^{-C\log^2\|y-x\|}$, 
\begin{equation}
\big| b([x], [y]) - \hat{b}(x,y) \big| \leq C'  h(\|y-x\|)\log^2\|y-x\|
	\leq \epsilon \|y-x\| .
\end{equation}
Lemma~\ref{lem:BApproxBHatInMean} in turn gives
\begin{equation}
	\big| \E\, b([x], [y]) - \E\, \hat{b}(x,y) \big| 
	\leq C' h(\|y-x\|)\log^2\|y-x\| \leq \epsilon \|y-x\| .
\end{equation}
Finally, by Proposition~\ref{lem:BConcentration}, with probability
at least $1-\texte^{-C \|y-x\|^{1/2} (h(\|y-x\|)^{-3/2}}$,
\begin{equation}
\label{E:3.43}
	\big|\hat{b}(x, y)-\E\,\hat{b}(x, y) \big| \leq \epsilon \|x-y\| \,.
\end{equation}
Combining \twoeqref{E:3.40}{E:3.43} and invoking the triangle inequality, we find that
\begin{equation}
\label{E:3.44}
\big| b([x], [y]) - \beta_p(y-x) \big | \leq 4 \epsilon
	\|y-x\| .
\end{equation}
holds with probability at least $1-\texte^{-C\log^2\|y-x\|}$. It remains to divide \eqref{E:3.44} by $\beta_p(y-x)$ and use that $\beta_p$ is equivalent to $\|\cdot\|$.
\end{proof}

\begin{proof}[Proof of Proposition~\ref{prop:GeomConcentration}]
Part~\ref{part:GC-1} is identical to Lemma~\ref{lem:BoundsOnOptPath}(3) so we only need to prove part~\ref{part:GC-2}. Note first that a qualitative version of this statement, one without an explicit decay estimate, can be proved without invoking the concentration bound in Theorem~\ref{thm:ConcentrationOfB}. Unfortunately, this would not be enough for our later use of this claim.

Fix~$x,y\in\Z^2$ with $\Vert x-y\Vert$ large, pick $\epsilon > 0$ and let $N:= \lceil 4\alpha/\epsilon\rceil$, where $\alpha$ is as in part~\ref{part:GC-1} of this proposition.
Define the sequence of vertices  $u_k := (1-kN^{-1}) x + kN^{-1} y$ where $k=0, \dots, N$. Then, for each $k=1, \dots, N$ we pick a path $\gamma_k \in \Gamma( [u_{k-1}], [u_k])$ that is minimal in a (previously assumed) complete ordering of~$\RR( [u_{k-1}], [u_k])$. Finally, set $\gamma_\epsilon := ((\dots(\gamma_1 *\gamma_2)\ast \dots )* \gamma_N)$. 

We claim that with probability at least $1-\texte^{-C\log^2 \|y-x\|}$ the path $\gamma_\epsilon$ satisfies the conditions on the left of \eqref{eqn:922}. Indeed, clearly $\gamma_\epsilon\in\RR([x],[y])$. Moreover, since $\|u_k - u_{k-1}\|$ is of order $N^{-1}\|y-x\|$, once $\|y-x\|$ is large enough we have 
\begin{equation}
\bigl|\frb(\gamma_k) - \beta_p(u_k - u_{k-1})\bigr| \leq \epsilon \|u_k - u_{k-1}\| / 3,
\end{equation} 
with probability at least $1-\texte^{-C\log^2\|y-x\|}$. This follows from the optimality of $\gamma_k$, Theorem~\ref{thm:ConcentrationOfB} and the 
equivalence of $\beta_p$ with~$\Vert\cdot\Vert$. The same arguments also show 
\begin{equation}
	\bigl| b([x], [y]) - \beta_p(y-x) \bigr| \leq \epsilon \|y-x\| / 3 ,
\end{equation}
with similar probability (albeit different constants). Summing over $k=1, \dots, N$, invoking sub-additivity of $\beta_p$ and Lemma~\ref{prop:MergingOfPaths}, with probability at least $1-\texte^{-C\log^2\|y-x\|}$,
\begin{equation}
\begin{split}
	\frb(\gamma_\epsilon) & \leq \sum_{k=1}^N \frb(\gamma_k) + C'N \\
	& \leq \sum_{k=1}^N \beta_p(u_k - u_{k-1}) + (\epsilon/3) \sum_{k=1}^N
		\|u_k - u_{k-1}\| + C'N \\
	& \leq \beta_p(y-x) + (\epsilon/2)\|y-x\|\, \leq \, b\bigl([x], [y]\bigr) + \epsilon \|y-x\| ,
\end{split}
\end{equation}
where we have used the positive homogeneity of $\beta_p$. This implies $\gamma_\epsilon \in \Gamma_\epsilon([x], [y])$. 

At the same time, by Lemma~\ref{lem:BoundsOnOptPath}(3) (or part~\ref{part:GC-1} of this proposition), with probability at least $1-\texte^{-C \epsilon\|y-x\|}$,
\begin{equation}
	|\gamma_k| \leq \alpha\|u_k-u_{k-1}\| < \epsilon\|y-x\|/2,\qquad k=1, \dots, N.
\end{equation}
Also, by Lemma~\ref{prop:DistanceToCInf} with probability of the same order $\|u_k - [u_k]\| \leq \epsilon \|y-x\|/4$. Together the last two statements imply that for $k = 1, \dots, N$, any $z \in \gamma_k$ and any $w \in \poly(u_{k-1}, u_k)$ we have $\|z-w\| < \epsilon \|y-x\|$ (notice that $N\epsilon>2$). This shows 
\begin{equation}
\label{E:3.49}
	\gamma_\epsilon \subseteq \poly(x,y) + \rmB_\infty(\epsilon \|y-x\|) \,.
\end{equation}
Since $\gamma_\epsilon$, viewed as curve in $\R^2$, is continuous and connects $[x]$ and $[y]$, which are at most $\epsilon \|y-x\|/4$ away from $x$ and $y$ respectively, it follows from \eqref{E:3.49} that 
\begin{equation}
	\poly(x,y) \subseteq \gamma_\epsilon + \rmB_\infty(\epsilon \|y-x\|).
\end{equation}
This shows that 
$d_{\rm H}(\gamma_\epsilon, \poly(x,y)) \leq \epsilon \|y-x\|$
with probability at least $1-\texte^{-C\epsilon\|y-x\|}$.  Since $\epsilon\|y-x\|\ge\log^2\|y-x\|$ as soon as $\|y-x\|$ is large, a union bound finishes the proof.
\end{proof}

\section{Approximating circuits by closed curves}
\label{sec4}\noindent
In this section we develop tools to describe the shape of large finite sets in the lattice using simple curves in~$\R^2$. This will then directly feed into our main results in Section~\ref{sec5}.

\subsection{Key propositions}
Recall the notion of a right-boundary circuit $\gamma$ and the correspondence with outer boundary interface~$\partial$ as detailed in Proposition~\ref{prop-duality}. By ``rounding the corners'' on edges on which the boundary interface reflects, $\partial$ can be identified with a simple closed curve --- i.e., a map from $[0,1]$ to~$\R^2$ which is injective on $[0,1)$ and has equal values at~$0$ and~$1$. We will write
\begin{equation}
\vol(\gamma):=\Z^2\cap\intr(\partial)
\end{equation}
to denote the ``filling'' of~$\gamma$, i.e., set of lattice points surrounded by the curve~$\partial$; see Fig.~\ref{fig1a}. Note that $\gamma\subset\vol(\gamma)$ whenever (which will be typical) the interface~$\partial$ associated with a right-most circuit~$\gamma$ is oriented counterclockwise. 

Our proofs require that we approximate $\vol(\gamma)$ by a set in~$\R^2$ whose boundary is a rectifiable simple closed curve~$\lambda$:

\begin{proposition}[Circuits to curves]
\label{lem:CircuitToLoopApprox}
Suppose~$p>p_\cc(\Z^2)$. For each $\epsilon > 0$ there is $C>0$ such that for all $R>1$ the following holds with probability at least $1-\texte^{-C\log^2 R}$: For any right-most circuit $\gamma$ which is oriented counterclockwise and obeys
\settowidth{\leftmargini}{(11111)}
\begin{enumerate}
\item $\gamma$ is open,
\item $\gamma \subseteq [-R,R]^2\cap\Z^2$,
\item $|\gamma| \geq R^{1/5}$,
\item $|\gamma| \leq |\vol(\gamma)|^{2/3}$,
\end{enumerate}
there is a simple closed curve $\lambda$ such that
\begin{enumerate}
   \item \label{part:CTLA-O1}
   $d_{\rm H}\bigl(\vol(\gamma), \intr(\lambda)\bigr) \leq 1+\epsilon \sqrt{|\vol(\gamma)|}$.
   \item \label{part:CTLA-O2}
   $\big| |\vol(\gamma)| - \Leb(\intr(\lambda)) \big|
       \leq \epsilon \bigl|\vol(\gamma)\bigr|$.
   \item \label{part:CTLA-O3}
   $\frb(\gamma) \geq (1-\epsilon) \len_{\beta_p}(\lambda)$.
\end{enumerate}
\end{proposition}

Note that, in~(1) above, we have invoked the natural embedding $\vol(\gamma)\subset\R^2$ to assign meaning to $d_{\rm H}((\vol(\gamma), \intr(\lambda))$. Our next claim tells how to go back from curves to circuits:

\begin{proposition}[Curves to circuits]
\label{lem:LoopToCircuitApprox}
Let $\lambda$ be any rectifiable simple closed curve in~$\R^2$ such that $\intr(\lambda)$ is convex and $R, \epsilon > 0$. Writing $\lambda_R := R\lambda$ (as a map $[0,1]\to\R^2$) let $\AA_{R,\epsilon,\lambda}$ denote the event that there is a counterclockwise-oriented right-most circuit $\gamma$ satisfying
\settowidth{\leftmargini}{(11111)}
\begin{enumerate}
	\item \label{part:LTCA-0}
	$\gamma$ is open,
	\item \label{part:LTCA-1}
		$d_{\rm H}\bigl(\vol(\gamma), \intr(\lambda_R)\bigr) \leq \epsilon R$,
	\item \label{part:LTCA-2}
		$\big||\vol(\gamma)| - \Leb(\intr(\lambda_R)) \big|
		\leq \epsilon R^2$,
	\item \label{part:LTCA-3}
		$\frb(\gamma) \leq (1+\epsilon) \len_{\beta_p}(\lambda_R)$.
\end{enumerate}
For each $p>p_\cc(\Z^2)$, $\epsilon > 0$ and $\lambda$ as above, there is $C > 0$ such that
\begin{equation}
\BbbP\bigl(\AA_{R,\epsilon,\lambda})\ge1-\texte^{-C\log^2R},\qquad R>1.
\end{equation}
\end{proposition}

A natural method to go between paths on the lattice and continuous curves is by way of polygonal approximations. These have been invoked already in various studies of two-dimensional ``Wulff construction'' in statistical mechanics  (Alexander-Chayes and Chayes~\cite{ACC}, Dobrushin, Koteck\'y, Shlosman~\cite{Book:DKS}, etc). However, the reliance on polygonal approximations is limited only to the proofs of the above propositions.

\subsection{Polygonal approximations}
Let~$\lambda$ be a rectifiable curve. For $r > 0$ we define its \textit{$r$-polygonal approximation} $\PP_r(\lambda)$ inductively as follows. Set $t_0 := 0$, $x_0 := \gamma(t_0)$ and, for $k=1,2, \dots$ until $t_k = 1$ define 
\begin{equation}
    \begin{split}
        t_k & := \inf \bigl\{t\in(t_{k-1},1]\colon\|\lambda(t) -x_{k-1}\|> r \bigr\} \wedge 1 , \\
        x_k & := \lambda(t_k).
    \end{split}
\end{equation}
Since~$\lambda$ is rectifiable, the process stops after a finite number of steps; i.e., there is~$N$ with
\begin{equation}
\label{eqn:961}
N \leq \bigl\lceil \len_\infty(\lambda) / r \bigr\rceil < \infty
\end{equation}
such that $t_N = 1$ and~$x_N=\lambda(1)$. We then set $\PP_r(\lambda):=\poly(x_0, \dots, x_N)$, i.e., $\PP_r(\lambda)$ is the concatenation of the line segments $\poly(x_i,x_{i+1})$, $i=0,\dots,N-1$. 

A few remarks are in order. First notice that 
\begin{equation}
\label{eqn:962}
	\|x_k - x_{k-1}\| = r,\quad k=0, \dots, N-1, \quad \text{and} \qquad 
		0 < \|x_N - x_{N-1}\| \leq r, 
\end{equation}
and then, by definition of the length of a curve,
\begin{equation}
\label{eqn:1061}
	\len_\rho\bigl(\PP_r(\lambda)\bigr) \leq \len_\rho(\lambda)
\end{equation}
for any norm~$\rho$ on~$\R^2$.
A slight complication is that $\PP_r(\lambda)$ may not be simple and so there could be several bounded connected components of~$\R^2\setminus\PP_r(\lambda)$. For a rectifiable closed curve~$\lambda$ in~$\R^2$, we thus introduce the notion of the \emph{hull of~$\lambda$} as follows: For~$x\not\in\lambda$ let $w_\lambda(x)$ denote the winding number of $\lambda$ around~$x$. Since~$\lambda$ is closed and rectifiable, $w_\lambda(x)\in\Z$ and so we can set
\begin{equation}
\hull(\lambda):=\lambda\cup\bigl\{x\not\in\lambda\colon w_\lambda(x)\text{ is odd}\bigr\}.
\end{equation}
It follows that $\hull(\lambda)$ is closed, connected and bounded.

\begin{lemma}
\label{lem:GoodnessOfPApprox}
Let $\lambda$ be any rectifiable curve.
\begin{enumerate}
	\item \label{part:GOPA-1}
		For all $r > 0$ we have $d_{\rm H}(\hull(\lambda), \, \hull({\PP}_r(\lambda))) \leq r$.
	\item \label{part:GOPA-2}
		If $\lambda$ is a closed curve, then for all $r>0$ and some~$C$ independent of~$\lambda$,
		\begin{equation}
		\label{E:4.6}
			\Leb \big(\hull(\lambda) \, \triangle \, \hull(\PP_r(\lambda) \big) 
				\leq C r (\len_\infty(\lambda) \vee r).
		\end{equation}
\end{enumerate}
\end{lemma}

\begin{proof}
Let $\lambda[a,b]$ denote the image of $[a,b]$ under~$\lambda$. Both parts follow from the observation
\begin{equation} 
	 \poly(x_{k-1},x_k) ,\,\, \lambda[t_{k-1}, t_k]\,\subseteq\, x_{k-1} + \rmB_\infty(r),\qquad k=1, \dots, N.
\end{equation}
This immediately gives~(1); for (2) we notice that the symmetric difference is covered by the union of sets $x_k+\rmB_\infty(r)$, $k=1,\dots,N$. In light of \eqref{eqn:961}, the bound \eqref{E:4.6} follows.
\end{proof}


As a next step, we will show that the non-simplicity of polygonal approximations can be readily overcome by a perturbation argument.

\begin{lemma}
\label{lemma-poly-to-simple}
For any closed polygonal curve $\lambda$, any $\epsilon>0$ and any norm~$\rho$ on~$\R^2$, there is a simple closed polygonal curve $\lambda'$ such that
\settowidth{\leftmargini}{(1111)}
\begin{enumerate}
\item[(1)]
$d_{\rm H}\bigl(\hull(\lambda),\intr(\lambda')\bigr)<\epsilon$,
\item[(2)]
$\Leb\bigl(\hull(\lambda)\,\triangle\,\intr(\lambda')\bigr)<\epsilon$,
\item[(3)]
$\len_\rho(\lambda')\le\len_\rho(\lambda)+\epsilon$.
\end{enumerate}
\end{lemma}

\begin{proofsect}{Proof}
The curve~$\lambda=\poly(x_0,\dots,x_n)$ is composed of linear segments $\poly(x_i,x_{i+1})$ meeting at vertices~$x_i$. By a limiting argument, we may assume that
\settowidth{\leftmargini}{(1111)}
\begin{enumerate}
\item[(a)] no two linear segments are parallel,
\item[(b)] each linear segment contains no other vertices than its endpoints,
\item[(c)] no more than two linear segments intersect at each point.
\end{enumerate}
Let~$z$ be a self-intersection point of~$\lambda$. Then there are two linear segments of~$\lambda$ that intersect at~$z$ and, by~(b), four components of $\R^2\setminus\lambda$ that meet at~$z$, two of which may be the same component. A little thought then reveals that two of these components have winding number even and two of them odd, with (necessarily) even components separated from each other by the odd ones and \emph{vice versa}. 

We will now introduce another polygonal line $\lambda'$ as follows. Let $0\le i<j \le n$ be such that the line segments intersecting at~$z$ are exactly $\poly(x_i,x_{i+1})$ and $\poly(x_j,x_{j+1})$. Now pick a point~$z'$ in the component of $\R^2\setminus\lambda$ that is adjacent to segments $\poly(x_i,z)$ and $\poly(x_j,z)$ and let~$z''$ be a similar point in the component adjacent to segments $\poly(z,x_{i+1})$ and $\poly(z,x_{j+1})$. Then set\begin{equation}
\lambda':=\poly(x_0,\dots,x_i,z',x_j,x_{j-1},\dots,x_{i+1},z'',x_{j+1},\dots,x_n).
\end{equation}
Notice that the sequence of points $x_{i+1},\dots,x_j$ is now run backwards; we have thus changed the orientation of one cycle in~$\lambda$.

For~$z'$ and~$z''$ close enough to~$z$, (a-c) above apply to~$\lambda'$ and $\lambda'$ is thus a polygonal line with $n+2$ vertices but one less intersection point than~$\lambda$. Moreover, $\R^2\setminus\lambda'$ has one fewer component than~$\R^2\setminus\lambda$ as two components $K_1$ and~$K_2$ of the $\R^2\setminus\lambda$ --- necessarily meeting at~$z$ with same parity of the winding number --- have ``joined'' to one, say~$K$, of $\R^2\setminus\lambda'$. Obviously, $K_1,K_2\subseteq K$ and the closure of $\hull(\lambda)\triangle\hull(\lambda')$ equals the closure of $K\setminus(K_1\cup K_2)$. Hence, given~$\epsilon>0$ there is~$\delta>0$ such that if $\|z-z'\|,\|z-z''\|<\delta$, then (1-3) hold. Proceeding similarly, we can gradually eliminate all intersection points of~$\lambda$ in a finite number of steps and thus prove the result.
\end{proofsect}

Having ensured that the polygonal approximation of a right-most circuit will ultimately lead to a simple curve, our next task will be to show that we can pass from a polygonal approximation of a closed curve to a right-most circuit. As it turns out, we will only need to do this for curves that arise as boundaries of convex sets. So it will be helpful to have:

\begin{lemma}
\label{lem:PolyApproxOfConvex}
Let $\lambda$ be a simple closed curve such that $\intr(\lambda)$ is convex. Then $\PP_r(\lambda)$ is simple for any $r$ which satisfies $0<r<\frac12\diam_\infty(\lambda)$. 
\end{lemma}

\begin{proof}
Assume $0<r<\frac12\diam_\infty(\lambda)$ and define~$\PP_r(\lambda)=\poly(x_0,\dots,x_N)$ as above. It is easy to check that $N\ge3$. Now consider a sequence of curves $\{\lambda_k\colon k=0,\dots,N\}$ where~$\lambda_k$ is obtained by concatenating $\poly(x_0,\dots,x_k)$ with $\lambda[t_k,1]$. Pick $k\in\{0,\dots,N-1\}$. We claim that if $\lambda_k$ is simple with $\intr(\lambda_k)$ convex, then the same holds for~$\lambda_{k+1}$. Indeed, convexity of $\intr(\lambda_k)$ implies
\begin{equation}
\label{E:4.8a}
\poly(x_k,x_{k+1})\cap\lambda\bigl([t_k,t_{k+1}]^\cc\bigr)=\emptyset
\end{equation}
because otherwise $\lambda((t_k,t_{k+1})^\cc)=\poly(x_k,x_{k+1})$ and thus $N=2$, contradicting our choice of~$r$. But when \eqref{E:4.8a} is in force, $\lambda_{k+1}$ is simple and $\intr(\lambda_{k+1})$ must be convex, as it is the intersection of two convex sets: $\intr(\lambda_k)$ and a closed half-plane whose boundary is the straight line containing $x_k$ and $x_{k+1}$. As $\lambda_0:=\lambda$ is simple with $\intr(\lambda)$ convex, the claim follows by induction.
\end{proof}

Having reduced the problem to $\lambda$ given by a simple polygonal curve, we observe:

\begin{lemma}
\label{lem:InversePolyApprox}
Let $\lambda:=\poly(x_0, \dots, x_N)$ be a simple closed polygonal curve in $\R^2$ and for~$R>0$ denote $\lambda_R:=R\lambda$. There is a constant~$C = C(\lambda) > 0$ such that the following holds: If for some $\delta>0$, $R \ge1 $, points $\tilde{x}_0, \dots, \tilde{x}_N \in \Z^2$ with $\tilde x_N=\tilde x_0$ and open paths $\gamma_k\in\RR(\tilde x_{k-1},\tilde x_k)$,
\begin{equation}
\label{E:4.8}
	d_{\rm H} \big(\gamma_k,\, \poly(Rx_{k-1}, Rx_k) \big) < \delta R
\end{equation}
holds for each $k=1,\dots,N$, then there is also an open right-most circuit $\gamma$ that obeys
\settowidth{\leftmargini}{(1111)}
\begin{enumerate}
\item[(1)]
$d_{\rm H}\bigl(\vol(\gamma), \intr(\lambda_R)\bigr) < 1+C\delta R$,  
\item[(2)]
$\bigl||\vol(\gamma)|-\Leb(\intr(\lambda_R))\bigr| < C\delta^2R^2$, 
\item[(3)]
$\frb(\gamma)\le\sum_{k=1}^N\frb(\gamma_k)+2N$.
\end{enumerate}
Moreover, if $\lambda$ is oriented counterclockwise, then so is the boundary interface associated with~$\gamma$.
\end{lemma}

\begin{proof}
Let $\lambda$ be as in the conditions of the lemma, set $\lambda_k := \poly(x_k ,x_{k+1})$
where $k=0, \dots, N-1$ and let $\epsilon := \min_k \len_\infty(\lambda_k)$. For any such $\lambda$ we may find $\delta > 0$ small enough such that
\settowidth{\leftmargini}{(1111)}
\begin{enumerate}
\item[(i)] $\big(\lambda_k + \rmB_\infty(\delta)\big) \, \cap \, \big(\lambda_j + \rmB_\infty(\delta)\big) = \emptyset$ if $|k-j| > 1$,
\item[(ii)]  $\big(\lambda_k + \rmB_\infty(\delta)\big) \, \cap \, \big(\lambda_{k+1} + \rmB_\infty(\delta)\big) \subseteq \big(x_{k+1} + \rmB_\infty(\epsilon / 3) \big)$ for all 
$k=0, \dots, N-1$,
\end{enumerate}
where all indices are {\it modulo} $N$. In this case, the complement of the set $\lambda + \rmB_\infty(\delta)$ consists of one finite and one infinite open connected components whose boundaries are simple closed curves, which we denote by $\lambda^-$ and $\lambda^+$, respectively.
Moreover, they satisfy 
\begin{equation}
\label{E:4.13}
	\intr(\lambda^-) \subseteq \intr(\lambda) \subseteq \intr(\lambda^+)
		\quad\text{and} \quad
	\lambda + \rmB_\infty(\delta) = \overline{\intr(\lambda^+)} \setminus \intr(\lambda^-) 
\end{equation}
and 
\begin{equation}
\label{E:4.14}
d_{\rm H} \big(\intr(\lambda^-),\, \intr(\lambda) \big) \leq \delta
	\quad\text{and} \quad
d_{\rm H} \big(\intr(\lambda^+),\, \intr(\lambda) \big) \leq \delta \,.
\end{equation}

We may choose $C = C(\lambda)$ large enough such that conclusions (1-3) will be trivially satisfied if $\delta$ is not small enough for (i) and (ii) to hold. Hence, we continue assuming that they do. Now let $\gamma_k$ for $k=1, \dots, N$ be the right most open paths in the conditions of the lemma. We shall construct $\gamma$ as follows. First set $\gamma' := (\dots ((\gamma_1 * \gamma_2) * \dots ) * \gamma_{N-1}$. This gives a right-most path from $x_0$ to $x_{N-1}$ whose vertices we enumerate as $\gamma' = (u_0, \dots, u_M)$. To close the circuit we need to carefully add the last path $\gamma_N = (v_0, \dots,  v_L)$.
To this end, we find the first vertex $u_{k'}$ in $\gamma'$ which lies in $R x_{N-1} + \rmB_\infty(R\epsilon / 3) \cap \gamma_N$ and then the last occurrence $v_{k}$ of this vertex in $\gamma_N$. Then, we also find the first vertex $v_{j}$ with $j > k$ such that $v_j$ lies in $(R x_0 + \rmB_\infty(R \epsilon / 3)) \cap \gamma'$ and similarly the last occurrence $u_{j'}$ of this vertex in $\gamma'$. We thus set $\gamma := (u_{j'}, \dots, u_{k'}, v_{k+1}, \dots, v_j)$ and claim that it satisfies conclusions (1-3) of the lemma.

Indeed, the construction ensures that $\gamma$ is a right-most circuit and, in view of condition \eqref{E:4.8}, that $\gamma \subseteq \lambda_R + \rmB_\infty(R\delta)$. At the same time, from (i--ii) above it follows that $\gamma$ passes through $Rx_k + \rmB_\infty(R\epsilon/3)$ in the order of increasing~$k$. This shows
\begin{equation}
\label{E:4.15}
	\intr(R \lambda^-) \subseteq \hull(\gamma) \subseteq \intr(R \lambda^+) \,.
\end{equation}
Combining this with \eqref{E:4.14} and the fact that $d_{\rm H} \big(\vol(\gamma),\, \hull(\gamma) \big) \leq 1$, conclusion (1) holds. Moreover, \eqref{E:4.15} together with \eqref{E:4.13} and
$\big| |\vol(\gamma)| - \hull(\gamma) \big| < C |\gamma|$
imply 
\begin{equation}
	\big| |\vol(\gamma)| - \Leb(\intr(\lambda_R)) \big| 
		\leq C \Leb \big(\lambda_R + \rmB_\infty(R\delta)\big) \leq C' R^2 \delta^2
			\len_\infty(\lambda) \leq C'' R^2 \delta^2 \,.
\end{equation}
This proves conclusion (2). Finally applying Lemma~\ref{prop:MergingOfPaths} for $\frb(\gamma')$ and a similar argument as in its proof for $\frb(\gamma)$ yields conclusion (3).
\end{proof}

\subsection{Proof of approximation claims}
We are now ready to prove Propositions~\ref{lem:CircuitToLoopApprox} and~\ref{lem:LoopToCircuitApprox}.

\begin{proofsect}{Proof of Proposition~\ref{lem:CircuitToLoopApprox}}
Let $\epsilon>0$ and, given $R$ large enough, set $r := \lceil R^{1/100} \rceil$. Let $\AA_{R,\epsilon}$ be the set of configurations~$\omega$ such that
\begin{equation}
\label{eqn:964}
	x,y \in \rmB_\infty(R) \cap \Z^2,\, \|y-x\| \geq r \quad \Longrightarrow \quad 
	b\bigl([x],[y]\bigr) \geq (1-\ffrac\epsilon2)\,\beta_p(y-x).
\end{equation}
and, for any simple path~$\gamma$ on~$\Z^2$,
\begin{equation}
\label{eqn:965}
	\gamma \subseteq \rmB_\infty(R)\cap\Z^2,\, |\gamma| \geq R^{1/5} ,\,\gamma \text{ is open}
		\quad \Longrightarrow \quad \gamma \subseteq \mbC^\infty .
\end{equation}
Using Theorem~\ref{thm:ConcentrationOfB} for \eqref{eqn:964}, and the exponential bound for the probability that $x,y$ are connected but not part of~$\mbC^\infty$ for \eqref{eqn:965}, we find that for each $p>p_\cc(\Z^d)$ there is $C>0$ such that
\begin{equation}
\BbbP(\AA_{R,\epsilon})\ge 1-\texte^{-C\log^2 R},\qquad R>1.
\end{equation}
Assuming that~$R$ is large enough, we will now prove that the claim in proposition holds for all~$\omega\in\AA_{R,\epsilon}$.

Let $\gamma := (z_0, \dots, z_M)$ be a right-most circuit satisfying the premises~(1-4) of the claim. Note that, by \eqref{eqn:965}, $\gamma\subset\mbC^\infty$. We may identify this circuit with a curve $\Gamma$ in~$\R^2$ by following (at linear speed) the edges of the path~$\gamma$. This permits us to consider the $r$-polygonal approximation, $\PP_r(\Gamma)$, of~$\Gamma$ ``started'' from~$z_0$. The curve $\PP_r(\Gamma)$ will ``almost'' satisfy conclusions~(1-3) of the lemma, except that it may not be simple. We shall therefore first show that conclusions~(1-3) hold for $\PP_r(\Gamma)$ with $\hull(\PP_r(\Gamma))$ in place of $\intr(\lambda)$ and then use Lemma~\ref{lemma-poly-to-simple} to extract a simple curve $\lambda'$ out of $\PP_r(\Gamma)$ for which (1-3) will hold \emph{verbatim}.

The premises~(3) and~(4) imply that $|\vol(\gamma)|\ge R^{3/10}$. Lemma~\ref{lem:GoodnessOfPApprox}(1) and $r\ll\epsilon R^{1/10}$ then show
\begin{equation}
\begin{aligned}
d_{\rm H}\bigl(\vol(\gamma),\hull(\PP_r(\Gamma))\bigr) 
&\leq 1 + d_{\rm H}\bigl(\hull(\Gamma),\hull(\PP_r(\Gamma))\bigr) 
\\
&\leq 1 + r \leq \epsilon \sqrt{|\vol(\gamma)|}, 
\end{aligned}
\end{equation}
i.e., conclusion~(1) holds. Using also Lemma~\ref{prop:DualPathLength}, we similarly get
\begin{equation}
\begin{split}
	\bigl| |\vol(\gamma)| - \Leb\bigl(\hull(\PP_r(\Gamma))\bigr) \bigr|
	 	& \leq 
	 \bigl| \Leb(\hull(\Gamma)) - \Leb\bigl(\hull(\PP_r(\Gamma))\bigr) \bigr| + |\gamma| \\
	& \leq C r |\gamma| 
	  \leq (1+\epsilon) |\vol(\gamma)| ,
\end{split}
\end{equation}
i.e.,~(2) holds as well. For~(3), notice that since $r$ is integer, it follows from the construction  that $\PP_r(\Gamma) = \poly(z_{l_0}, \dots, z_{l_N})$ for some integers $0=:\ell_0<\ell_1<\dots<\ell_N\le M$. We now use that $z_i\in\mbC^\infty$ for all~$i$ and also $\|z_{l_k} - z_{l_{k-1}}\| = r$ for $k =1, \dots, N-1$ and \eqref{eqn:964} to get
\begin{equation}
\label{eqn:981}
\begin{aligned}
\frb(\gamma)&\geq \sum_{k=1}^{N-1} \frb\bigl(\gamma([l_{k-1},l_k])\bigr) 
	\ge\sum_{k=1}^{N-1} b(z_{l_{k-1}}, z_{l_k}) \\
		&\ge(1-\ffrac{\epsilon}{2}) \sum_{k=1}^{N-1} \beta_p(z_{l_k} - z_{l_{k-1}}) 
		\\&=(1-\ffrac{\epsilon}{2}) \bigl[  
			\len_{\beta_p}(\PP_r(\Gamma)) - \beta_p(z_{l_N} - z_{l_{N-1}}) \bigr] .
\end{aligned}
\end{equation}
The bounds $\len_{\beta_p}(\PP_r(\Gamma))\ge C'\text{diam}(\gamma)\ge C'R^{\frac1{10}}$ and $\beta_p(z_{l_N} - z_{l_{N-1}})\le C''r$ imply that also conclusion~(3) holds.

To complete the proof, we now use Lemma~\ref{lemma-poly-to-simple} with $\rho:=\beta_p$ and $\epsilon$ small enough to extract a simple closed curve~$\lambda'$. The triangle inequality then ensures that conclusions (1-3) hold for~$\lambda'$ with $2\epsilon$ instead of~$\epsilon$.
\end{proofsect}

\begin{proofsect}{Proof of Proposition~\ref{lem:LoopToCircuitApprox}}
Let $\lambda$ be a simple curve with a convex interior and $\epsilon > 0$ be given. We may assume that~$\lambda$ is oriented counterclockwise. By Lemma~\ref{lem:GoodnessOfPApprox} and Lemma~\ref{lem:PolyApproxOfConvex},
for $r > 0$ small enough the polygonal approximation $\PP_r(\lambda) = \poly(x_0, \dots, x_N)$ of~$\lambda$ is simple and satisfies
\begin{equation}
\label{eqn:976}
	d_{\rm H}\bigl(\intr(\lambda),\intr(\PP_r(\lambda))\bigr) \leq \epsilon
\end{equation}
and
\begin{equation}
\label{eqn:975}
	\bigl| \Leb(\intr(\lambda)) - \Leb(\intr(\PP_r(\lambda))) \bigr| 
		\leq \epsilon.
\end{equation}
Thanks to Proposition~\ref{prop:GeomConcentration}(\ref{part:GC-2}) and Theorem~\ref{thm:ConcentrationOfB}, once $R$ is sufficiently large then the following holds with probability at least $1-\texte^{-C\log^2 R}$: For each $k=1, \dots, N$ there exists an open path $\gamma_k\in\RR([R x_{k-1}],[R x_k])$ such that
\begin{equation}
\label{eqn:973}
	\frb(\gamma_k) \leq (1+\epsilon)R \beta_p(x_k - x_{k-1})
\end{equation}
and
\begin{equation}
\label{eqn:974}
	d_{\rm H}\bigl(\gamma_k, \poly(Rx_{k-1}, Rx_k)\bigr) \leq \epsilon R.
\end{equation}
Applying Lemma~\ref{lem:InversePolyApprox}, we extract a right-most open circuit~$\gamma$ satisfying conclusions~\hbox{(1-3)} of Lemma~\ref{lem:InversePolyApprox} with~$\delta:=\epsilon$. In conjunction with \twoeqref{eqn:976}{eqn:975}, this readily yields the conclusions~(1-3) of the proposition but for a re-scaling of~$\epsilon$ by a constant that might depend on $\lambda$.

To get also conclusion~(4), we note that, thanks to \eqref{eqn:973} and Lemma~\ref{lem:InversePolyApprox}(3), for $R$ large enough, the path~$\gamma$ obeys
\begin{equation}
\begin{split}
	\frb(\gamma) \ & \leq \ \sum_{k=1}^N\frb(\gamma_k) + 2N 
		\ \leq \ (1+\epsilon) R\sum_{k=1}^N \beta_p(x_k - x_{k-1}) + 2N \\
		& \leq \ (1+2\epsilon) R\,\len_{\beta_p} \bigl(\PP_r(\lambda)\bigr)
		\ \leq \ (1+2\epsilon) R\,\len_{\beta_p} (\lambda).
\end{split}
\end{equation}
Here the last two inequalities follow by the definition of the length of a curve and the fact that~$\PP_r(\lambda)$ is a polygonal approximation of~$\lambda$. Resetting $\epsilon$ the proof is done.
\end{proofsect}

\section{Proof of Main Theorems}
\label{sec5}\noindent
In this section we will ultimately prove the main theorems of this work. However, before we get down to actual proofs we need some preliminary considerations.

\subsection{Symmetries of the Wulff shape}
As our first preliminary step, we need to check a few basic properties of the Wulff shape. We already observed that that $\beta_p$ is symmetric with respect to reflections through the coordinate axes and diagonals in~$\R^2$. 
As a consequence of these symmetries, we can derive:

\begin{lemma}
\label{lem:W_p_properties}
Let $p>p_\cc(\Z^2)$ and let $W_p$ be as in \eqref{eqn:4.5}. Then
\settowidth{\leftmargini}{(1111)}
\begin{enumerate}
\item[(1)]
$W_p$ is a compact convex set, contains the origin and has a non-empty interior.
\item[(2)] 
$W_p$ is symmetric with respect to reflections through coordinate axes of~$\R^2$ and the diagonal line $\{(x_1, x_2) \in \R^2 \colon  x_1 = x_2 \}$.
\item[(3)]
$W_p$ is the unit ball in the dual norm $\beta_p'(y) := \sup \{x\cdot y\, \colon  x \in \R^2, \, \beta_p(x) \leq 1\}$.
\item[(4)]
There is $r$ with $1/2 \le \, r \, \le1/\sqrt 2$ such that the normalized Wulff shape~$\widehat W_p$ obeys
\begin{equation}
\label{E:5.6}
\rmB_1(r)\subseteq \widehat{W}_p \subseteq \rmB_\infty(r).
\end{equation}
\end{enumerate}
\end{lemma}

\begin{proof}
By definition, $W_p$ is closed, convex and contains the origin. The boundedness and non-triviality of the interior follows from the uniform boundedness of $\hat n\mapsto\beta_p(\hat n)$ away from~$0$ and~$\infty$ for $\hat n$ on the unit circle. The symmetries of~$W_p$ are inherited from those of $\beta_p$, as shown in Lemma~\ref{prop:SymmetryForBeta}. This proves parts~(1) and~(2); for~(3) we first check that $\beta'_p$ is a norm and then note that $\{x \colon  \beta'_p(x) \leq 1\}$ is just a rewrite of the definition of $W_p$.

For part~(4), let  $r := \max \{x_1 \geq 0 \colon  (x_1, 0) \in\widehat W_p\}$ and note that, by the symmetries in part~(2), in addition to $(+r,0)\in\widehat W_p$ also $(-r, 0), (0, +r), (0, -r) \in\widehat W_p$. Then $\widehat W_p \supseteq \rmB_1(r)$ by convexity. We claim that also $\widehat W_p \subseteq \rmB_\infty(r)$. For if not, then for some $x := (x_1, x_2)$ with $x_1 > r$ we would have $\hat{n}\cdot x \leq \beta_p$ for all $\hat{n}$ on the unit circle. The symmetries of $\beta_p$ would then imply
\begin{equation}
\hat{n}\cdot (2x_1, 0) =\hat{n}\cdot \bigl( (x_1, x_2) + (x_1, -x_2)\bigr) \leq 2\beta_p(\hat n),\qquad  \Vert\hat n\Vert_2=1,
\end{equation}
and so $(x_1, 0) \in\widehat W_p$, in contradiction with the definition of~$r$. As $2r^2=\Leb(\rmB_1(r))\le\Leb(\widehat W_p)=1$, we must have $r\le1/\sqrt2$. Also $r\ge\ffrac12$ as $\Leb(\rmB_\infty(\ffrac12))=1$.
\end{proof}

\subsection{Percolation preliminaries}
Our next set of preliminary considerations deals with percolation. The following is a slightly stronger version of Proposition~1.2 of Benjamini and Mossel~\cite{benjamini2003mixing}.

\begin{lemma}
\label{prop:GiantIsCInf}
For each~$p>p_\cc(\Z^2)$ there is $C>0$ such that for all $n \geq 1$, with probability at least $1-\texte^{-C\log^2n}$, for any $n'\le n - \log^2 n$,
\begin{equation}
    \label{eqn:1011}
    \mbC^\infty  \cap \rmB_\infty(n') =
    \mbC^n \cap \rmB_\infty(n').
\end{equation}
\end{lemma}

\begin{proof}
The statement in \cite{benjamini2003mixing} claims that \eqref{eqn:1011} holds with probability tending to $1$ as $n \to \infty$. However, a careful look at the proof reveals that
the probability of the complement of the event \eqref{eqn:1011} is exponentially small in~$n-n'$. A direct duality argument is also possible.
\end{proof}

Our next lemma provides uniform bounds on the density of the infinite cluster inside sufficiently large subsets of the lattice.

\begin{lemma}
\label{lem:CInfDensity} 
Let~$p>p_\cc(\Z^2)$. For each $\epsilon > 0$, there is $C>0$ such that for all $R > 1$ with probability at least $1-\texte^{-C\log^2 R}$, if $\gamma$ is any right-most circuit satisfying
\settowidth{\leftmargini}{(1111)}
\begin{enumerate}
\item $\gamma \subseteq \rmB_\infty(R)$,
\item $|\vol(\gamma)| \geq \log^{20} R$,
\item $|\gamma| \leq |\vol(\gamma)|^{2/3}$,
\end{enumerate}
then
\begin{equation}
\label{E:5.9}
    \left|\frac{|\vol(\gamma) \cap \mbC^\infty|}
    	{|\vol(\gamma)|} -\theta(p) \right| < \epsilon.
\end{equation}
\end{lemma}

\begin{proof}
Fix $\epsilon > 0$, let $R>1$ and set $r := \lfloor \log^2 R \rfloor$. 
For $u \in \Z^2$ let $B_u := \bigl((2r)u + [-r,r)^2 \bigr) \cap \Z^2$. Note that $|B_u| = 4r^2 \leq C \log^4 R$ and that $\{B_u \colon  u \in \Z^2\}$ form a partition of $\Z^2$.
By Durrett and Schonmann~\cite[Theorems~2 and~3]{durrett1988large} and a simple union bound, there is~$C>0$ such that
\begin{equation}
\label{eqn:1234}
\AA_R:=\left \{
  u \in \Z^2,\,\, B_u \cap \rmB_\infty(R) \neq \emptyset 
  \quad \Rightarrow \quad
  \left| \frac{\big| \mbC^\infty \, \cap \, B_u \big|}
    {|B_u|} - \theta(p) \right| < \epsilon
\right \}
\end{equation} 
occurs with probability $\BbbP(\AA_R)\ge1-\texte^{-C\log^2 R}$ for some $C>0$.

Now suppose that~$\AA_R$ occurs and let $\gamma$ be a circuit satisfying conditions (1--3). Abbreviate $A:= \vol(\gamma)$, set $A_r$ to the union of all boxes $B_u$ for which $B_u\subseteq A$ and let~$A^r$ denote the union of all boxes $B_u$ for which $A \cap B_u \neq \emptyset$. Clearly, $A_r \subseteq A \subseteq A^r$.
At the same time, $A^r \setminus A_r$ is the collection of all boxes $B_u$ that have at least one vertex in $A$ and at least one vertex outisde $A$. Such boxes must include a vertex of $\gamma$ and so there are at most $|\gamma| < |A|^{2/3}$ of them. As condition~(2) implies $|A|^{2/3}|B_u|\ll|A|$ once~$R$ sufficiently large, we have
\begin{equation}
  |A^r| \leq |A| + |A|^{2/3}|B_u| \leq |A|(1+\epsilon)
    \quad\text{and} \quad
  |A_r| \geq |A| - |A|^{2/3}|B_u| \geq |A|(1-\epsilon).
\end{equation}
Hence, on the event $\AA_R$,
\begin{equation}
  |\,A \cap \mbC^\infty| \leq |\,A^r \cap \mbC^\infty| \leq 
  (\theta_p + \epsilon) |A^r| \leq \theta_p |A| (1 + C\epsilon)
\end{equation}
and 
\begin{equation}
  |\,A \cap \mbC^\infty| \geq |\,A_r \cap \mbC^\infty| \geq 
  (\theta_p - \epsilon) |A_r| \geq \theta_p |A| (1 - C\epsilon)
\end{equation}
are in force. Writing $\epsilon$ for~$C\epsilon$, we get \eqref{E:5.9}.
\end{proof}

\subsection{Key propositions}
We are now ready start addressing the proofs of Theorems~\ref{thm:Profile} and~\ref{thm:Cheeger}. The key arguments for both of these are the same and so we encapsulate them into propositions. The first of these will lead to lower bounds on the isoperimetric profile and the Cheeger constant:

\newcounter{stepcounter}
\newcommand\STEP[1]{\refstepcounter{stepcounter}\smallskip\noindent\textit{STEP~\thestepcounter\ }(#1):\ }

\begin{proposition}[Lower bound]
\label{prop-5.6}
Let~$p>p_\cc(\Z^2)$ and pick $\zeta>\ffrac25$. For each~$\epsilon>0$ there is~$C>0$ such that for all~$n>1$ the event 
\begin{equation}
\label{eqn:1012}
\left\{U\subset\mbC^\infty\cap\rmB_\infty(n),\,\,\text{\rm connected},\,\,|U|\ge n^{\zeta}\,\,\Rightarrow\,\,
\frac{\bigl|\partial_{\mbC^\infty} U |}{\sqrt{|U|}}\geq (1-\epsilon) \theta_p^{-1/2} \varphi_p\right\}
\end{equation}
occurs with probability at least $1-\texte^{-C\log^2 n}$.
\end{proposition}

\begin{proofsect}{Proof}
Let $\epsilon > 0$, $n> 1$ and let $\UU_n$ be the collection of all connected subsets of $\mbC^\infty\cap\rmB_\infty(n)$ with $|U|\ge n^\zeta$. We will occasionally regard~$U\in\UU_n$ as a graph $G_U$ obtained by restricting~$\mbC^\infty$ to vertices in~$U$. The proof comes in a sequence of six steps.

\STEP{Identifying a right-boundary circuit}
The graph~$G_U$ is a connected subgraph of the planar graph~$\Z^2$ and so there is a unique boundary interface~$\partial$ that separates~$G_U$ from the unique infinite connected component of~$\Z^2\setminus G_U$. By Proposition~\ref{prop-duality}, this boundary interface then defines a unique right-boundary path~$\gamma\subseteq U$, oriented counterclockwise, such that
\begin{equation}
\label{E:5.15}
U\subseteq\vol(\gamma)\quad\text{and}\quad
\frb(\gamma)\le|\partial_{\mbC^\infty}U|.
\end{equation}

\STEP{A lower bound on~$|\gamma|$}
Clearly, $|\gamma|\ge C|\partial|$ for some~$C>0$ and $|\partial|$ is proportional to the length of the simple closed curve in~$\R^2$ that~$\partial$ can be identified with. Since $|U|\le|\vol(\gamma)| \leq C \ \Leb(\intr(\partial))$,  the isoperimetric inequality in~$\R^2$ tells us
\begin{equation}
|\gamma|\ge C|\partial|\ge C'\bigr|\Leb(\intr(\partial))\bigr|^{\ffrac12}\ge C''\sqrt{|U|}.
\end{equation}
Using that~$|U|\ge n^\zeta$ we have $|\gamma|\ge n^{1/5}$ once~$n$ is large enough. 

\STEP{An upper bound on~$|\gamma|$}
Once we know~$|\gamma|\ge n^{1/5}\gg\log^2n$, Proposition~\ref{prop:RenormalizedLD} implies that, for some~$\alpha>0$, on an event with probability at least $1-\texte^{-C'\log^2 n}$, 
\begin{equation}
\label{E:5.17}
	\frb(\gamma) > \alpha |\gamma|,\qquad U\in\UU_n.
\end{equation}
Consequently, if we have $|\gamma| \ge C \sqrt{|U|}$, then \eqref{E:5.15} and \eqref{E:5.17} yield 
\begin{equation}
|\partial_{\mbC^\infty} U| \geq \alpha C \sqrt{|U|}.
\end{equation}
The conclusion in \eqref{eqn:1012} then follows once~$C$ is so large that~$\alpha C\ge\theta_p^{-1/2}\varphi_p$. We may thus assume without loss that $|\gamma| \leq C \sqrt{|U|}$ for some~$C>0$. Let~$\UU'_n$ be the subset of~$\UU_n$ where this holds.

\STEP{Approximation of~$\gamma$ by a curve}
Consider the event that for every right-most circuit~$\gamma$ arising from a $U\in\UU_n'$, there is a simple closed curve $\lambda$ satisfying
\settowidth{\leftmargini}{(1111)}
\begin{enumerate}
   \item[(1)]
   $d_{\rm H}\bigl(\vol(\gamma), \intr(\lambda)\bigr) \leq 1+\epsilon \sqrt{|\vol(\gamma)|}$.
   \item[(2)]
   $\big| |\vol(\gamma)| - \Leb(\intr(\lambda)) \big|
       \leq \epsilon \bigl|\vol(\gamma)\bigr|$.
   \item[(3)]
   $\frb(\gamma) \geq (1-\epsilon) \len_{\beta_p}(\lambda)$.
\end{enumerate}
Since $\gamma \subseteq \rmB_\infty(n)$, $|\gamma|\ge n^{1/5}$ and $|\vol(\gamma)|^{2/3}\ge|U|^{2/3}\ge C\sqrt{|U|}\ge|\gamma|$, Proposition~\ref{lem:CircuitToLoopApprox} with $R:= n$ and $\epsilon$ as above implies that this event occurs with probability at least $1-\texte^{-C\log^2 n}$.

\STEP{Comparing volumes}
Lemma~\ref{lem:CInfDensity} with $R:=2n$ and the aforementioned bounds $|\gamma|\ge 
\log^{20}n$ and $|\gamma|\le |\vol(\gamma)|^{2/3}$ tell us that, on an event of probability at least $1-\texte^{-C\log^2n}$, 
\begin{equation}
\label{eqn:1016}
	|U| \leq \bigl|\mbC^\infty \cap \vol(\gamma)\bigr|
		\leq (1+\epsilon) \theta_p \bigl|\vol(\gamma)\bigr|
\end{equation}
for any set~$U\in\UU_n'$. 

\STEP{Wrapping up}
The definition of~$\varphi_p$ (see \eqref{eqn:IsoOfBeta}) yields
\begin{equation}
	\len_{\beta_p}(\lambda) \geq \varphi_p \sqrt{\Leb(\intr(\lambda))}.
\end{equation}
On the event where the conclusions of STEPS~3,~4, and~5 apply --- which has probability at least $1-\texte^{-C\log^2n}$ --- we then have
\begin{equation}
\begin{aligned}
|\partial_{\mbC^\infty} U| 
\geq \frb(\gamma) 
&\geq (1-\epsilon) \len_{\beta_p}(\lambda)
\geq (1-\epsilon)\, \varphi_p \,\sqrt{\Leb(\intr(\lambda))}  \\
& \geq (1-\epsilon)^{3/2} \,\varphi_p\, \sqrt{|\vol(\gamma)|}
\ge\frac{(1-\epsilon)^{3/2}}{\sqrt{1+\epsilon}} \,\varphi_p\,\theta_p^{-\ffrac12}\,\sqrt{|U|}
\end{aligned}
\end{equation}
for all $U\in\UU_n'$. A simple adjustment of~$\epsilon$ then yields the desired claim. 
\end{proofsect}

Next we will formulate a proposition that will ultimately yield desired upper bounds on the isoperimetric characteristics in our main theorems. 

\begin{proposition}[Upper bound]
\label{prop-5.7}
Let~$p>p_\cc(\Z^2)$ and, given~$\epsilon>0$ and~$n>1$, let $\AA_n'$ be the event that there exists a set~$U$ satisfying:
\begin{equation}
\label{E:5.22a}
\mbC^\infty\cap\rmB_1(\ffrac n4)\subseteq U\subseteq\mbC^\infty\cap\rmB_\infty(n), \text{\rm\ connected},
\end{equation}
\begin{equation}
\label{E:5.23a}
\frac{1}2(1-\sqrt\epsilon)\bigl|\mbC^\infty\cap\rmB_\infty(n)\bigr|\le|U|\le\frac12\bigl|\mbC^\infty\cap\rmB_\infty(n)\bigr|
\end{equation}
and
\begin{equation}
\label{E:5.24a}
\frac{|\partial_{\mbC^\infty} U |}{\sqrt{|U|}}\leq (1+\epsilon) \theta_p^{-1/2} \varphi_p.
\end{equation}
Then for each~$\epsilon>0$ there is $C>0$ such that
\begin{equation}
\BbbP(\AA_n')\ge1-\texte^{-C\log^2 n},\qquad n>1.
\end{equation}
\end{proposition}

\begin{proofsect}{Proof}
Without loss of generality, it suffices to prove this for~$\epsilon$ sufficiently small and~$n$ sufficiently large. We again proceed along a sequence of steps.

\setcounter{stepcounter}{0}
\STEP{Finding a candidate for~$U$}
Let $\hat{\gamma}_p$ be the Wulff curve given by \eqref{eqn:4.51}. By definition
\begin{equation}
\label{E:5.23}
\Leb(\intr(\hat{\gamma}_p))=1\quad\text{and}\quad\len_{\beta_p}(\hat{\gamma}_p) = \varphi_p.
\end{equation}
From Lemma~\ref{lem:W_p_properties} we know that $\intr(\hat{\gamma}_p)$ is convex. Thus, we may use  Proposition~\ref{lem:LoopToCircuitApprox} with~$\epsilon$ as above, $\lambda := \hat{\gamma}_p$ and
\begin{equation}
\label{E:5.24}
R:= \sqrt{2}\,(1-\epsilon^{2/3})\,n ,
\end{equation}
to conclude that with probability at least $1-\texte^{-C \log^2 n}$ there is a right-most circuit $\gamma$ oriented counterclockwise and satisfying 
\settowidth{\leftmargini}{(1111)}
\begin{enumerate}	
\item[(1)] $\gamma$ is open,
\item[(2)] $d_{\rm H}\bigl(\vol(\gamma), R\widehat{W}_p\bigr) \leq \epsilon R$,
\item[(3)] $\big||\vol(\gamma)| - R^2\big| \leq \epsilon R^2$,
\item[(4)] $\frb(\gamma) \leq (1+\epsilon) R\varphi_p$.
\end{enumerate}
We now set $U:=\vol(\gamma)\cap\mbC^\infty$ and prove that~$U$ obeys \twoeqref{E:5.22a}{E:5.24a} for~$\epsilon$ small enough.

\STEP{Trivial conditions}
First of all, $U$ is connected because $\gamma\subset U$ and any vertex in~$U$ is connected by an open path (i.e., its path to infinity in~$\mbC^\infty$) to~$\gamma$, which is open as well. Also $\mbC^\infty\cap\rmB_1(\ffrac n4)\subseteq U$ for small~$\epsilon$ thanks to conclusion~(2) and Lemma~\ref{lem:W_p_properties}(4). As $(1+\epsilon)R/\sqrt2\le n$ when~$\epsilon$ is small enough, Lemma~\ref{lem:W_p_properties}(4) also shows 
\begin{equation}
\label{E:5.25}
\vol(\gamma)\subseteq(1+C\epsilon)R\,\widehat W_p\subseteq\rmB_\infty\bigl((1+ C \epsilon)R/\sqrt2\bigr)\subseteq B_\infty(n).
\end{equation}
In particular, $U\subseteq\mbC^\infty\cap\rmB_\infty(n)$ and so \eqref{E:5.22a} holds.

\STEP{Comparison of volumes}
We proceed to check \eqref{E:5.23a}. Since $|\gamma|$ is of order~$n$ and $|\vol(\gamma)|$ is of order~$n^2$, Lemma~\ref{lem:CInfDensity} is at our disposal once $n$ is large enough. This yields
\begin{equation}
\label{E:5.26}
\bigl||U|-\theta_p|\vol(\gamma)|\bigr|\le\epsilon\bigl|\vol(\gamma)\bigr|
\end{equation}
on an event of probability at least $1-\texte^{-C\log^2 n}$. On the same event, we also have
\begin{equation}
\label{E:5.31}
(1-\epsilon)\theta_p\le\frac{|\mbC^\infty\cap\rmB_\infty(n)|}{|\rmB_\infty(n)|}\le(1+\epsilon)\theta_p.
\end{equation}
Thanks to property~(3) above, from \twoeqref{E:5.26}{E:5.31} and \eqref{E:5.24} we get
\begin{multline}
\qquad
|U|\le(\theta_p+\epsilon)\bigl|\vol(\gamma)\bigr|\le(\theta_p+\epsilon)(1+\epsilon)R^2
\\
\le \frac{\theta_p+\epsilon}{\theta_p}\,\frac{1+\epsilon}{1-\epsilon}(1-\epsilon^{2/3})\frac{2n^2}{|\rmB_\infty(n)|}
\bigl|\mbC^\infty\cap\rmB_\infty(n)\bigr|.
\qquad
\end{multline}
For~$\epsilon$ small and~$n$ large, this yields $|U|\le\frac12|\mbC^\infty\cap\rmB_\infty(n)|$. Similarly, we derive
\begin{multline}
\label{E:5.29}
\qquad
|U|\ge(\theta_p-\epsilon)\bigl|\vol(\gamma)\bigr|\ge(\theta_p-\epsilon)(1-\epsilon)R^2
\\
\ge\frac{\theta_p-\epsilon}{\theta_p}\frac{1-\epsilon}{1+\epsilon}(1-\epsilon^{2/3})^2\frac{2n^2}{|\rmB_\infty(n)|}
\bigl|\mbC^\infty\cap\rmB_\infty(n)\bigr|
\qquad
\end{multline}
and so, for~$\epsilon$ small and~$n$ large, $|U|\ge\frac12(1-\sqrt\epsilon)^2|\mbC^\infty\cap\rmB_\infty(n)|$.

\STEP{Surface-to-volume ratio}
As to \eqref{E:5.24a}, since $U$ includes all vertices in $\mbC^\infty$ which lie in $\vol(\gamma)$, by Proposition~\ref{prop-duality} every edge in $\partial_{\mbC^\infty} U$ must also be in $\partial^+ \gamma$ (in some orientation). Hence, using also conclusion~(4) above,
\begin{equation}
|\partial_{\mbC^\infty} U|\le\frb(\gamma)\le(1+\epsilon)R\varphi_p.
\end{equation}
But $|U|\ge(\theta_p-\epsilon)(1-\epsilon)R^2$ by  \eqref{E:5.29} and so $|\partial_{\mbC^\infty} U|\le(1+C\epsilon)\theta_p^{-1/2}\varphi_p\sqrt{|U|}$ for some~$C>0$. Adjusting~$\epsilon$, the inequality \eqref{E:5.24a} follows.
\end{proofsect}

\subsection{Proofs of limit values and shapes}
We are now ready to establish the almost sure limits of the (properly scaled) isoperimetric profile and Cheeger constant. We will do this by proving (matching) upper and lower bounds. Below, we will use the notation $\Phi_{\mbC^\infty}(n) := \Phi_{\mbC^\infty, 0}(n)$ and $\BbbP^0(-):=\BbbP(-|0\in\mbC^\infty)$. Note that $\BbbP^0\ll\BbbP$.

\begin{proofsect}{Proof of Theorems~\ref{thm:Profile}--\ref{thm:Cheeger}, lower bounds}
Fix~$\epsilon>0$, pick~$\zeta\in(\ffrac25,\ffrac12)$ and let~$\AA_n$ denote the event in \eqref{eqn:1012}. We claim that there is (a random) $n_0$ with $\BbbP(n_0<\infty)=1$ such that $\AA_n$ occurs, \eqref{E:5.31} holds and $\mbC^n\subset\mbC^\infty$ is valid for all $n\ge n_0$. This follows in light of the summability of $\BbbP(\AA_n^\cc)$ on~$n$ and the Borel-Cantelli lemma, the Spatial Ergodic Theorem (or Lemma~\ref{lem:CInfDensity}), and Lemma~\ref{prop:GiantIsCInf}, respectively. 

As a first step we note that we may restrict attention to sets $U$ with $|U|\ge n^\zeta$. Indeed, since~$\mbC^\infty$ is infinite and connected and $\mbC^n \subset \mbC^\infty$, we have $|\partial_{\mbC^\infty}U|\ge1$. So if $|U|<n^\zeta$, then
\begin{equation}
\label{E:5.34}
\frac{|\partial_{\mbC^\infty}U|}{|U|}\ge n^{-\zeta}\gg n^{-1/2}
\end{equation}
and so $U$ cannot contribute to $\Phi_{\mbC^\infty}(n)$ and $\widetilde\Phi_{\mbC^n}$ provided the limits take the anticipated values. (Alternatively, we can also invoke \eqref{E:1.4}.)

Since any finite connected $U\subset\mbC^\infty$ with $0\in U$ and $|U|\le n$ will obey $U\subseteq\mbC^\infty\cap\rmB_\infty(n)$, the inequality in~$\AA_n$ applies to all minimizers of~$\Phi_{\mbC^\infty}(n)$ (once~$n\gg1$). To get a similar conclusion for
minimizers $U$ of~$\widetilde\Phi_{\mbC^n}$, we note that, without loss of generality, we can assume that $U$ is connected. Indeed, if it is not, we can replace it by one if its connected components $U'$ which is also a minimizer and satisfies \eqref{E:5.34} and 
$U'\subset\mbC^\infty\cap\rmB_\infty(n)$ by definition.

We can thus take a connected $U\subset\mbC^\infty\cap\rmB_\infty(n)$ with $|U|\ge n^\zeta$ that minimizes either $\Phi_{\mbC^\infty}(n)$ or~$\widetilde\Phi_{\mbC^n}$ and use it to derive a lower bound on these quantities. As~$\AA_n$ occurs,
\begin{equation}
\label{E:5.35}
\frac{|\partial_{\mbC^\infty}U|}{|U|}=\frac1{\sqrt{|U|}}\frac{|\partial_{\mbC^\infty}U|}{\sqrt{|U|}}
\ge\frac1{\sqrt{|U|}}(1-\epsilon)\theta_p^{-1/2}\varphi_p.
\end{equation}
For the isoperimetric profile, $|U|\le n$ then implies
\begin{equation}
\liminf_{n\to\infty}\,\,n^{1/2}\,\Phi_{\mbC^\infty}(n)\ge(1-\epsilon)
\theta_p^{-1/2}\varphi_p,
\qquad\BbbP^0\text{\rm-a.s.},
\end{equation}
while for the Cheeger constant we instead use $|U|\le\frac12|\mbC^\infty\cap\rmB_\infty(n)|$ and \eqref{E:5.31} to get
\begin{equation}
\liminf_{n\to\infty}\,\,n\,\widetilde\Phi_{\mbC^n}\ge\frac{1-\epsilon}{\sqrt{1+\epsilon}}
\frac1{\sqrt2}\,\theta_p^{-1}\varphi_p,
\qquad\BbbP\text{\rm-a.s.}.
\end{equation}
Letting~$\epsilon\downarrow0$, the almost-sure lower bounds in \eqref{E:1.6} and \eqref{E:1.8} are proved.
\end{proofsect}

\begin{proofsect}{Proof of Theorems~\ref{thm:Profile}--\ref{thm:Cheeger}, upper bounds}
Fix~$\epsilon>0$ and let~$\AA'_n$ be the event in Proposition~\ref{prop-5.7}. Using Borel-Cantelli, there is~$n_0$ such that $\AA'_n$ occurs and \eqref{E:5.31} 
holds for all~$n\ge n_0$ almost surely with respect to both $\BbbP$ and $\BbbP^0$. Using Lemma~\ref{prop:GiantIsCInf}, we may also assume that for $n\ge n_0$, \eqref{eqn:1011} holds $\BbbP$-a.s.

Suppose $n\ge n_0$ and let~$U$ be a set satisfying the properties defining~$\AA_n'$. From \eqref{E:5.23a} and \eqref{E:5.31} (and $\epsilon<1$) we get
\begin{equation}
\label{E:5.38}
2(1-\sqrt\epsilon)^2\theta_pn^2\le|U|\le2(1+\epsilon)\theta_pn^2
\end{equation}
whereby \eqref{E:5.24a} implies
\begin{equation}
\label{E:5.39}
\frac{|\partial_{\mbC^\infty}U|}{|U|}=\frac1{\sqrt{|U|}}\,\frac{|\partial_{\mbC^\infty}U|}{\sqrt{|U|}}\le
\frac1{\sqrt{|U|}}\,(1+\epsilon)\theta_p^{-1/2}\varphi_p
\le\frac1n\,\frac1{\sqrt2}\frac{1+\epsilon}{1-\sqrt\epsilon}\theta_p^{-1}\varphi_p.
\end{equation}
In light of \eqref{eqn:1011} we have $U\subset\mbC^{m}$ and $|U|\le\frac12|\mbC^{m}|$ for any $m\ge n+\log^2 n$. It follows that
\begin{equation}
\limsup_{n\to\infty}\,n\,\widetilde\Phi_{\mbC^n}\le\frac1{\sqrt2}\frac{1+\epsilon}{1-\sqrt\epsilon}\,\theta_p^{-1}\varphi_p,
\qquad\BbbP\text{-a.s.}
\end{equation}
Similarly, since \eqref{E:5.22a} guarantees that on the event $\{0 \in \mbC^\infty\}$ the origin will be included in $U$ and since it is also connected, we can use it to bound $\Phi_{\mbC^\infty}(r)$ via \eqref{E:5.39} whenever $r$ is between the values of the right-hand side of \eqref{E:5.38} evaluated for $n$ and $n+1$. This yields
\begin{equation}
\limsup_{r\to\infty}\,r^{1/2}\,\Phi_{\mbC^\infty}(r)\le\frac{(1+\epsilon)^2}{1-\sqrt\epsilon}\,\theta_p^{-1/2}\varphi_p,
\qquad\BbbP^0\text{-a.s.}
\end{equation}
Taking~$\epsilon\downarrow0$, the upper bounds in \eqref{E:1.6} and \eqref{E:1.8} hold. 
\end{proofsect}

\begin{proofsect}{Proof of Theorem~\ref{thm-limit-value}}
The characterization of the limit value was used to establish Theorems~\ref{thm:Profile}--\ref{thm:Cheeger}. The symmetries of the norm have been shown in Lemma~\ref{prop:SymmetryForBeta}. 
\end{proofsect}

Moving over to limit shapes, we first note a simple corollary of the preceding proofs. Recall that $\UU_{\mbC^\infty}(n)$, $\UU_{\mbC^n}$ are the set of minimizers of $\Phi_{\mbC^\infty}(n)$, resp. $\widetilde{\Phi}_{\mbC^n}$.

\begin{corollary}[Tightness of used volume]
\label{cor-volume}
Let $p>p_\cc(\Z^d)$. For each $\delta>0$ there is $n_0=n_0(\delta)$ with $\BbbP(n_0<\infty)=1$ such that the following holds for all $n \geq n_0$:
\begin{equation}
\inf_{U\in\hat\UU_{\mbC^\infty}(n)}\frac{|U|}n\ge(1-\delta)
\end{equation}
and
\begin{equation}
\label{E:5.43}
\min_{U\in\hat\UU_{\mbC^n}}\frac{|U|}{|\rmB_\infty(n)|}\ge\frac12\theta_p(1-\delta)
\end{equation}
Moreover, each $U\in\hat\UU_{\mbC^\infty}(n)$ and $U\in\hat\UU_{\mbC^n}$ is connected and
\begin{equation}
\label{E:5.44c}
\frac{|\partial_{\mbC^\infty} U|}{\sqrt{|U|}}\le(1+\delta)\theta_p^{-1/2}\varphi_p \,.
\end{equation}
\end{corollary}

\begin{proofsect}{Proof}
Suppose the event in~\eqref{eqn:1012} occurs for $n\ge n_0$. If, for arbitrarily large $n$'s,
a connected set~$U\ni0$ satisfies $|U| <  n(1-\delta)$, then \twoeqref{E:5.34}{E:5.35} with $\epsilon$ such that $\sqrt{1-\delta} < 1-\epsilon$, show that $|\partial_{\mbC^\infty}U|/|U|$ exceeds the established a.s.\ limit value of $n^{1/2}\Phi_{\mbC^\infty}(n)$. So~$U$ cannot be a minimizer of~$\Phi_{\mbC^\infty}(n)$ for~$n$ sufficiently large. The same argument applies to connected minimizers $U\in\hat\UU_{\mbC^n}$. However, by~\eqref{E:5.31}, $|U|\le\frac12(1+\delta)\theta_p|\rmB_\infty(n)|$ for all minimizers when~$n\ge n_0$ and so (for~$\delta$ small)~$U$ must be connected.

To see that \eqref{E:5.44c} holds for any minimizer once~$n$ is large enough, note that the opposite inequality would imply, using the first part of \eqref{E:5.35}, that the already established limit values in \eqref{E:1.6} and \eqref{E:1.8} are larger by a factor of at least $(1+\delta)$ than what they should be.
\end{proofsect}

\begin{proofsect}{Proof of Theorems~\ref{thm-1.6}--\ref{thm-1.7}}
The limits \eqref{E:1.15} and \eqref{E:1.18} are now a direct consequence of Corollary~\ref{cor-volume}, the natural constraints on the minimizers and also \eqref{E:5.31}. 
Let $\epsilon > 0$ be arbitrarily small and suppose that $n$ is large enough such that
the event in \eqref{eqn:1012} for some $\zeta \in (\frac25, \frac12)$, the conclusions of Corollary~\ref{cor-volume} with $\delta := \epsilon$ and Lemma~\ref{lem:CInfDensity} with $R=n$ hold for such $\epsilon$. Pick $U$ to be any minimizer of either of the two problems. Setting 
\begin{equation}
\widetilde W_n:=\begin{cases}
\theta_p^{-1/2} n^{1/2} \widehat{W}_p,\qquad&\text{for Theorem~\ref{thm-1.6}},
\\*[1mm]
\sqrt{2} n  \widehat{W}_p,\qquad&\text{for Theorem~\ref{thm-1.7}},
\end{cases}
\end{equation} 
it will be enough to show that
\begin{equation}
\label{E:5.44}
	d_{\rm H}(U, \xi + \widetilde W_n) \leq C \sqrt{\Leb(\widetilde W_n)} \epsilon
\end{equation}
for some $\xi \in \R^2$ and $C = C(p) > 0$

Consider the right-boundary circuit~$\gamma\subseteq U$ as in STEP~1 of the proof of Proposition~\ref{prop-5.6}. By the inequality \eqref{E:5.44c},
\begin{equation}
\label{E:5.45}
\frb(\gamma)\le|\partial_{\mbC^\infty} U|\le|U|^{1/2}(1+\epsilon)\theta_p^{-1/2}\varphi_p.
\end{equation}
Proposition~\ref{lem:CircuitToLoopApprox} implies that there is a simple closed curve~$\lambda$ approximating~$\gamma$ so that properties (1-3) in STEP~4 of the proof of Proposition~\ref{prop-5.6} hold.
Thanks to \eqref{E:5.31} and $|\vol(\gamma)|\gg1$ is large,
\begin{equation}
\label{E:5.46}
|U|\le(1+\epsilon)\theta_p\bigl|\,\vol(\gamma)\bigr|\le\frac{1+\epsilon}{1-\epsilon}\,\theta_p\,\Leb(\intr(\lambda)).
\end{equation}
Since also by property~(3) in STEP~4 of the above proof
\begin{equation}
\label{E:5.46a}
\frb(\gamma)\ge(1-\epsilon)\len_{\beta_p}(\lambda)
\end{equation}
the scaled version~$\hat\lambda$ of~$\lambda$ normalized so that $\Leb(\intr(\hat\lambda))=1$ satisfies
\begin{equation}
\label{E:5.47}
\len_{\beta_p}(\hat\lambda)
\le\frac{(1+\epsilon)^{3/2}}{(1-\epsilon)^{3/2}}\varphi_p,
\end{equation}
i.e., $\hat\lambda$ is a near-minimizer of the Wulff variational problem.
The generalized Bonnesen inequality \eqref{E:1.15a} then gives 
\begin{equation}
\label{E:5.49a}
d_{\rm H}(\intr(\hat\lambda),\, \xi' + \widehat{W}_p)\le C\epsilon,
\end{equation}
for some $\xi' \in \R^2$ and $C>0$.

To get from this a bound on the distance between the Wulff shape and $U$, we use the triangle inequality. Setting $\xi := \sqrt{\Leb(\intr(\lambda))} \xi' $, the left-hand side of \eqref{E:5.44} is bounded above by
\begin{equation}
\label{E:5.49}	
\begin{split}
	d_{\rm H} \big(U, \vol(\gamma) \big) \ &+ \ 
	d_{\rm H} \big(\vol(\gamma), \intr(\lambda) \big) \ + \ 
	\sqrt{\Leb(\intr(\lambda))}\,\, d_{\rm H} \big(\intr(\hat{\lambda}), \xi' + \widehat{W}_p \big) \\
	 &+ \ d_{\rm H} \big(\xi + \sqrt{\Leb(\intr(\lambda))} \widehat{W}_p, \xi + \widetilde W_n \big)	\,.
\end{split}
\end{equation}
We shall show that each term is bounded above by $C \epsilon \sqrt{\Leb(\widetilde W_n)}$ where $C > 0$ does not depend on $n$, $\epsilon$ or $U$. This will validate \eqref{E:5.44} as needed. 

First, notice that~\eqref{E:5.46} gives a lower bound on $\Leb(\intr(\lambda))$. To get the opposite bound, we use the definition of $\varphi_p$ together with \eqref{E:5.46a} and~\eqref{E:5.45}
\begin{equation}
\label{E:5.51}
	\Leb(\intr(\lambda)) \leq \frac{\len_{\beta_p}^2(\lambda)}{\varphi_p^2} \leq 
		|U| \frac{(1+\epsilon)^2}{(1-\epsilon)^2} \theta_p^{-1} \,.
\end{equation}
The restriction on the size of $U$ and~\eqref{E:5.31} imply that 
\begin{equation}
\label{E:5.52}
	|U| \leq (1+\epsilon) \theta_p \Leb(\widetilde W_n)
\end{equation}
Combining this with~\eqref{E:5.49a}, \eqref{E:5.51} and STEP~4 of the proof of Proposition~\ref{prop-5.6}, the two middle terms in~\eqref{E:5.49} are bounded above by $C \epsilon \sqrt{\Leb(\widetilde W_n)}$. 

As for the last term in~\eqref{E:5.49}, it is bounded by
\begin{equation}
\label{E:5.53}
	\Bigl| \sqrt{\Leb(\intr(\lambda))} - \sqrt{\Leb(\widetilde W_n)}\, \Bigr| \diam \big( \widehat W_p \big) \,.
\end{equation}
In order to estimate this difference we use the upper and lower bounds on $\Leb(\intr(\lambda))$ in~\eqref{E:5.46} and~\eqref{E:5.51} in conjunction with the upper and lower bounds on $|U|$ given by~\eqref{E:5.52} and Corollary~\ref{cor-volume}. This yields again the bound $C\epsilon \Leb(\widetilde W_n)$ for \eqref{E:5.53}.

Finally, let $r:=d_{\rm H}(U,\vol(\gamma))$. As $U \subseteq \vol(\gamma)$, there must be an~$x\in\vol(\gamma)$ such that the box~$x+\rmB_\infty(r)$ has no intersection with~$U$. But once~$r\ge\log^{10} n$, Lemma~\ref{lem:CInfDensity} implies that
\begin{equation}
|U|\le(\theta_p+\epsilon)\bigl|\vol(\gamma)\bigr|-(\theta_p-\epsilon)\bigl|\rmB_\infty(r)\bigr|
\end{equation}
Plugging this into \eqref{E:5.45} instead of \eqref{E:5.46} and applying $\len_{\beta_p}(\hat\lambda)\ge\varphi_p$ yields $|\rmB_\infty(r)|\le C\epsilon|U|$. In light of 
\eqref{E:5.52}, this is only possible if $r \leq C\epsilon \Leb(\widetilde W_n)$. 
\end{proofsect}

\section*{Acknowledgments}
\noindent
This research has been partially supported by 
NSF grant DMS-1106850, NSA grant H98230-11-1-0171, GA\v CR project P201-11-1558, ERC grant StG 239990,
ISF grant 817/09 and the ESF project ``Random Geometry of Large Interacting Systems and Statistical Physics'' (RGLIS). 
We wish to thank Itai Benjamini for presenting us with Conjecture~\ref{con1} and to G\'abor Pete for many remarks and discussions at various stages of this work.

\end{document}